\documentclass[11pt,reqno]{amsart}

\usepackage{amsfonts}
\usepackage{amsmath}
\usepackage{mathrsfs}
\usepackage{amssymb, mathtools}
\usepackage{color, amscd, array, graphicx, mathdots, bm, amsfonts, epsfig, tikz, subcaption}
\usepackage{amsthm}
\usepackage{geometry}
\usepackage{caption}
\usepackage{blindtext}
\usepackage[toc, page]{appendix}
\usepackage{hyperref}
\usepackage{pgfplots}
\pgfplotsset{compat=1.18}

\theoremstyle{plain}
\newtheorem{theorem}{Theorem}[section]
\newtheorem{corollary}[theorem]{Corollary}
\newtheorem{lemma}[theorem]{Lemma}
\newtheorem{proposition}[theorem]{Proposition}

\theoremstyle{remark}
\newtheorem{remark}[theorem]{Remark}

\newcommand{\bE}{\mathbb{E}}

\newcommand{\bP}{\mathbb{P}}

\newcommand{\bR}{\mathbb{R}}
\newcommand{\bS}{\mathbb{S}}

\newcommand{\bZ}{\mathbb{Z}}

\newcommand{\cB}{\mathcal{B}}

\newcommand{\cF}{\mathcal{F}}

\newcommand{\cH}{\mathcal{H}}

\newcommand{\cS}{\mathcal{S}}
\newcommand{\cT}{\mathcal{T}}

\newcommand{\cZ}{\mathcal{Z}}

\newcommand{\kH}{\mathfrak{H}}

\newcommand{\ks}{\mathfrak{s}}

\newcommand{\ky}{\mathfrak{y}}
\newcommand{\kz}{\mathfrak{z}}

\newcommand{\bfH}{\mathbf{H}}

\newcommand{\bfr}{\mathbf{r}}
\newcommand{\bfs}{\mathbf{s}}

\newcommand{\bfy}{\mathbf{y}}

\newcommand{\1}{\mathbf{1}}

\newcommand{\wick}{:\! \exp \!:}

\begin{document}

\title[Feynman-Kac approximation for PAM]{Discrete Feynman-Kac approximation for parabolic Anderson model using random walks}

\author{Panqiu Xia}

\address{School of Mathematics, Cardiff University, Abacws, Senghennydd Road, Cathays, Cardiff, Wales, UK, CF24 4AG}
\email{\href{mailto:xiap@cardiff.ac.uk}{xiap@cardiff.ac.uk}}

\author{Jiayu Zheng}
\thanks{J. Zheng is supported by NSFC grant 11901598 and Guangdong Characteristic
Innovation Project No. 2023KTSCX163. }

\address{Faculty of Computational Mathematics and Cybernetics, Shenzhen MSU-BIT University, Shenzhen, Guangdong, China, 518172}

\email{\href{mailto:jyzheng@smbu.edu.cn}{jyzheng@smbu.edu.cn}}

\date{\today}

\begin{abstract}
In this paper, we introduce a natively positive approximation method based on the Feynman-Kac representation using random walks, to approximate the solution to the one-dimensional parabolic Anderson model of Skorokhod type, with either a flat or a Dirac delta initial condition. Assuming the driving noise is a fractional Brownian sheet with Hurst parameters $H \geq \frac{1}{2}$ and $H_* \geq \frac{1}{2}$ in time and space, respectively, we also provide an error analysis of the proposed method. The error in $L^p (\Omega)$ norm is of order 
\[
O \big(h^{\frac{1}{2}[(2H + H_* - 1) \wedge 1] - \epsilon}\big), 
\]
where $h > 0$ is the step size in time (resp. $\sqrt{h}$ in space), and $\epsilon > 0$ can be chosen arbitrarily small. This error order matches the H\"older continuity of the solution in time with a correction order $\epsilon$, making it `almost' optimal. Furthermore, these results provide a quantitative framework for convergence of the partition function of directed polymers in Gaussian environments to the parabolic Anderson model. 
\end{abstract}



\maketitle

\normalsize

\section{Introduction}

In the context of stochastic partial differential equations (SPDEs), the parabolic Anderson model (PAM) serves as a continuum version of the Anderson localization, which is a fundamental model in condensed matter physics. 
In this paper, we focus on the positive approximation for the solution to PAM. 
 Consider the following equation, 
\begin{align}\label{pam} \tag{PAM}
 \frac{\partial}{\partial t} u = \frac{1}{2} \Delta u + u \diamond \dot{W}, 
\end{align}
where $\Delta$ stands for the Laplace operator in space; $\dot{W}$ is a centred Gaussian noise defined on a complete probability space $(\Omega, \cF, \bP)$, that is the (formal) space-time derivative of a fractional Brownian sheet with Hurst parameters $H \geq \frac{1}{2}$ and $H_* \geq \frac{1}{2}$ in time and space, respectively;
and $\diamond$ denotes the Wick product, indicating the corresponding stochastic integral is in the Skorokhod sense. 

The lattice approximation (cf. \cite{pa-98-gyongy, pa-99-gyongy}) is a classic approach to numerical solutions for SPDEs. However, applying the lattice approximation for \eqref{pam} presents two issues. First, the Wick product presenting in \eqref{pam} is difficult to approximate numerically, except in the white-in-time case, i. e., $H = \frac{1}{2}$, as only in this special case does the Wick product coincide with the ordinary product. Second, even in the white-in-time case, the lattice approximation does not guarantee positivity of the numerical solutions, which is a key property when the equation start at a positive initial condition (cf. \cite{ap-19-chen-huang, sto-91-mueller, cmp-17-gubinelli-perkowski}). 
To resolve the first issue, one applicable approach is the Wick-Malliavin approximation (cf. \cite{cit-06-luo}) based on the chaos expansion to the solution. 
However, the Wick-Malliavin approximation still cannot resolve the issue related to the positivity of the numerical solution. 

This paper aims to present a natively positive approximation for \eqref{pam} based on its Feynman-Kac representation. This approach to numerical solutions for (deterministic) PDEs is well-established and is often seen as a specific example of the quantum Monte Carlo method (cf. \cite{cambridge-16-gubernatis-kawashima-werner}). In contrast, the study of the Feynman-Kac representation for solutions to SPDEs was first initiated in \cite{jsp-95-bertini-cancrini}, where the driving noise is $1$-dimensional space-time white noise, and the formula is expressed in a generalised sense through the Wick renormalisation. Subsequently, similar problems for equations driven by fractional Brownian sheets were studied in \cite{ap-11-hu-nualart-song, ap-12-hu-lu-nualart}. 

However, extending the Feynman-Kac approach to numerical solutions to SPDEs is still a relatively unexplored area. A primary challenge lies in the fact that, as detailed in \cite{jsp-95-bertini-cancrini}, under Wick renormalisation, an `$\infty$' term emerges in the exponent of the Feynman-Kac representation, when the equation is driven by $1$-dimensional space-time white noise. This `$\infty$' leads to an infinite variance, which persists in the statistical error when employing Monte Carlo methods, rendering it an unreliable approximation. A parallel obstruction occurs for equation \eqref{pam} driven by a fractional Brownian sheet with Hurst parameters $H \geq \frac{1}{2}$ and $H_* \geq \frac{1}{2}$ provided that $2 H_* + H \leq 2$. By contrast, if $2 H + H_* > 2$, the Wick renormalisation term is finite (cf. \cite[Theorems 3.1 and 7.2]{ap-11-hu-nualart-song}). Under this more restrictive condition, both the Skorokhod and Stratonovich solutions to \eqref{pam} exist, and the Monte Carlo approximation becomes applicable.

Consequently, 
alternative methods beyond the Monte Carlo approach are necessary for approximating the expectation. The core idea in this paper is discretising the Brownian motion as simple random walks, enabling explicit calculation of the resulting expectation. To the best of our knowledge, this is the first numerical study to employ the Feynman-Kac scheme using random walks to \eqref{pam}. 
To successfully implement the Feynman-Kac approximation using random walks, as presented in this paper, we must address the following two questions:
\begin{itemize}
 \item Do the approximate solutions converge to the true solution?
 
 \item If so, what is the rate of convergence? 
\end{itemize} 

The first question appears straightforward to confirm in terms of convergence in distribution. Indeed, in the study of directed polymers in random environments, this has been extensively investigated within a more general framework (cf. \cite{ap-14-alberts-khanin-quastel, jems-16-caravenna-sun-zygouras, ejp-23-chen-gao, spa-20-rang, ejp-24-rang-song-wang}). 
However, the strong ($L^2(\Omega)$) convergence, which is expected in the scenario of numerical solutions, remains unknown. By comparing the chaos expansions of the approximate and true solutions, the $L^2(\Omega)$ convergence depends on estimating the difference between the `densities' of rescaled simple random walks and Brownian motion in corresponding Hilbert space associated to the driving noise (see Subsection \ref{ss_malliavin}). 

Unfortunately, the density of rescaled simple random walks does not always converge. This can be illustrated with a simple argument. Consider a rescaled simple random walk $\sqrt{h} S_{\lfloor t /h\rfloor}$ with parameter $h > 0$. It is well-known that as $h \downarrow 0$, $\sqrt{h} S_{\lfloor t /h\rfloor}$ converges in distribution to a Brownian motion. Fix $t > 0$ and fix $h > 0$ sufficiently small. Assume, as intended, that the densities of $\sqrt{h} S_{\lfloor t /h\rfloor}$ and $B_t$ are close. Here, the density of $\sqrt{h} S_{\lfloor t /h\rfloor}$ refers to a function $f$ on $\bR_{\geq 0} \times \bR$ defined as 
\[
f_h (t, x) = \frac{1}{\sqrt{h}}\bP \Big(S_{\lfloor t /h\rfloor} = \big\lfloor x/\sqrt{h} \big\rfloor\Big). 
\]
Now, increase $t$ slightly to $t'$ so that $\lfloor t'/h \rfloor = \lfloor t/h \rfloor + 1$. Since Brownian motion has H\"{o}lder continuous densities, $B_t$ and $B_{t'}$ have very similar densities. In contrast, the densities of $\sqrt{h} S_{\lfloor t /h\rfloor}$ and $\sqrt{h} S_{\lfloor t' /h\rfloor}$ differ drastically, as they are supported on disjoint sets.

To address this issue, instead of using $\sqrt{h} S_{\lfloor t /h\rfloor}$, we can approximate Brownian motion with $\sqrt{h} \lfloor\frac{1}{2} S_{\lfloor 4 t /h\rfloor}\rfloor$. Then, one can prove that the densities $\widetilde{f}_h (t, x)$ of $ \sqrt{h} \lfloor\frac{1}{2}S_{\lfloor 4 t /h\rfloor}\rfloor$ form a convergent sequence in the underlining Hilbert space as $h \downarrow 0$.  This allows for a precise estimation of the error in $L^2(\Omega)$. An illustration of the issue and our proposed resolution is provided in Figure \ref{fig:overall}, while a more precise construction, assuming a flat initial condition, is provided in the next subsection.

\begin{figure}
  \centering
  \begin{subfigure}[b]{0.48\textwidth}
  \begin{tikzpicture}
\begin{axis}[
    width=\textwidth,
    height=0.55\textwidth, 
    domain=-3:3,
    samples=200,
    axis lines=left, 
    legend pos=outer north east,
    legend style={
        at={(0.5,-0.15)}, 
        anchor=north,
        legend columns=2,
        /tikz/every even column/.append style={column sep=0.5cm},
        font=\small,
    },
]

\addplot[
    thick,
    blue,
]
{1 / sqrt(2*pi) * exp(-x^2/2)};

\addplot[
    red,
    samples=1000,
    const plot,
    mark=none
]
coordinates {

(-3, 0.0030) (-2.846, 0.0030)
(-2.846, 0) (-2.5298, 0)
(-2.5298, 0.0308) (-2.2135, 0.0308)
(-2.2135, 0) (-1.897, 0) 
(-1.897, 0.1389) (-1.5811, 0.1389)
(-1.5811, 0) ( -1.2649, 0)
(-1.2649, 0.3705) (-0.9486, 0.3705)
(-0.9486, 0) (-0.6324, 0)
(-0.6324, 0.6485) (-0.3162, 0.6485)
(-0.3162, 0) (0, 0)
(0, 0.7782) (0.3162, 0.7782)
(0.3162, 0) (0.6324, 0)
(0.6324, 0.6485) (0.9486, 0.6485)
(0.9486, 0) (1.2649, 0)
(1.2649, 0.3705) (1.5811, 0.3705)
(1.5811, 0) (1.897, 0)
(1.897, 0.1389) (2.2135, 0.1389)
(2.2135, 0) (2.5298, 0)
(2.5298, 0.0308) (2.846, 0.0308)
(2.846, 0) (3, 0)
};
\end{axis}
\end{tikzpicture}
    \caption{Comparison of $p_1$ and $f_{0.1}(1, \cdot)$.}
    \label{fig:sub1}
  \end{subfigure}
  \hfill
  \begin{subfigure}[b]{0.48\textwidth}
   \begin{tikzpicture}
\begin{axis}[
    width=\textwidth,
    height=0.55\textwidth, 
    domain=-3:3,
    samples=200,
    axis lines=left, 
    legend pos=outer north east,
    legend style={
        at={(0.5,-0.15)}, 
        anchor=north,
        legend columns=2,
        /tikz/every even column/.append style={column sep=0.5cm},
        font=\small,
    },
]

\addplot[
    red,
    samples=1000,
    const plot,
    mark=none,
    forget plot
]
coordinates {

(-3, 0.0030) (-2.846, 0.0030)
(-2.846, 0) (-2.5298, 0)
(-2.5298, 0.0308) (-2.2135, 0.0308)
(-2.2135, 0) (-1.897, 0) 
(-1.897, 0.1389) (-1.5811, 0.1389)
(-1.5811, 0) ( -1.2649, 0)
(-1.2649, 0.3705) (-0.9486, 0.3705)
(-0.9486, 0) (-0.6324, 0)
(-0.6324, 0.6485) (-0.3162, 0.6485)
(-0.3162, 0) (0, 0)
(0, 0.7782) (0.3162, 0.7782)
(0.3162, 0) (0.6324, 0)
(0.6324, 0.6485) (0.9486, 0.6485)
(0.9486, 0) (1.2649, 0)
(1.2649, 0.3705) (1.5811, 0.3705)
(1.5811, 0) (1.897, 0)
(1.897, 0.1389) (2.2135, 0.1389)
(2.2135, 0) (2.5298, 0)
(2.5298, 0.0308) (2.846, 0.0308)
(2.846, 0) (3, 0)
};
\addplot[
    magenta,
    dashed,
    samples=1000,
    const plot,
    mark=none,
    forget plot
]
coordinates {

(-3, 0.0030) (-2.846, 0.0030)
(-2.846, 0.0308) (-2.2135, 0.0308)
(-2.2135, 0.1389) (-1.5811, 0.1389)
(-1.5811, 0.3705) (-0.9486, 0.3705)
(-0.9486, 0.6485) (-0.3162, 0.6485)
(-0.3162, 0.7782) (0.3162, 0.7782)
(0.3162, 0.6485) (0.9486, 0.6485)
(0.9486, 0.3705) (1.5811, 0.3705)
(1.5811,  0.1389) (2.2135, 0.1389)
(2.2135,  0.0308) (2.846, 0.0308)
(2.846, 0.0030) (3, 0.0030)
};
\end{axis}
\end{tikzpicture}
    \caption{Fill the $0$-intervals by their right neighbourhoods.}
    \label{fig:sub2}
  \end{subfigure}  
  \hfill
  \begin{subfigure}[b]{0.48\textwidth}
\begin{tikzpicture}
\begin{axis}[
    width=0.9\textwidth,
    height=0.55\textwidth, 
    domain=-3:3,
    samples=200,
    axis lines=left, 
    legend pos=outer north east,
    legend style={
        at={(0.5,-0.15)}, 
        anchor=north,
        legend columns=2,
        /tikz/every even column/.append style={column sep=0.5cm},
        font=\small,
    },
]

\addplot[
    magenta,
    dashed,
    samples=1000,
    const plot,
    mark=none,
    forget plot
]
coordinates {

(-3.000, 0.0015) (-2.846, 0.0015)
(-2.846, 0.0154) (-2.2135, 0.0154)
(-2.2135, 0.0694) (-1.5811, 0.0694)
(-1.5811, 0.1852) (-0.9486, 0.1852)
(-0.9486, 0.3242) (-0.3162, 0.3242)
(-0.3162, 0.3891) (0.3162, 0.3891)
(0.3162, 0.3242) (0.9486, 0.3242)
(0.9486, 0.1852) (1.5811, 0.1852)
(1.5811, 0.0694) (2.2135, 0.0694)
(2.2135, 0.0154) (2.846, 0.0154)
(2.846, 0.0015) (3.000, 0.0015)
};

\addplot[
    white,
    dashed,
    samples=1000,
    const plot,
    mark=none,
    forget plot
]
coordinates {
(-0.3162, 0.7782) (0.3162, 0.7782)
};

\end{axis}
\end{tikzpicture}    
\caption{Compress the height by one-half.}
    \label{fig:sub3}
  \end{subfigure}
   \hfill
  \begin{subfigure}[b]{0.48\textwidth}
\begin{tikzpicture}
\begin{axis}[
    width=0.9\textwidth,
    height=0.55\textwidth, 
    domain=-3:3,
    samples=200,
    axis lines=left, 
    legend pos=outer north east,
    legend style={
        at={(0.5,-0.15)}, 
        anchor=north,
        legend columns=2,
        /tikz/every even column/.append style={column sep=0.5cm},
        font=\small,
    },
]

\addplot[
    thick,
    blue,
]
{1 / sqrt(2*pi) * exp(-x^2/2)};

\addplot[
    magenta,
    dashed,
    samples=1000,
    const plot,
    mark=none,
    forget plot
]
coordinates {

(-3.000, 0.0015) (-2.846, 0.0015)
(-2.846, 0.0154) (-2.2135, 0.0154)
(-2.2135, 0.0694) (-1.5811, 0.0694)
(-1.5811, 0.1852) (-0.9486, 0.1852)
(-0.9486, 0.3242) (-0.3162, 0.3242)
(-0.3162, 0.3891) (0.3162, 0.3891)
(0.3162, 0.3242) (0.9486, 0.3242)
(0.9486, 0.1852) (1.5811, 0.1852)
(1.5811, 0.0694) (2.2135, 0.0694)
(2.2135, 0.0154) (2.846, 0.0154)
(2.846, 0.0015) (3.000, 0.0015)
};

\addplot[
    white,
    dashed,
    samples=1000,
    const plot,
    mark=none,
    forget plot
]
coordinates {
(-0.3162, 0.7782) (0.3162, 0.7782)
};
\end{axis}
\end{tikzpicture}
\caption{Comparison of $p_1$ and $\widetilde{f}_{0.1}(1, \cdot)$.}
    \label{fig:sub4}
\end{subfigure}
  \caption{Comparison of densities of Brownian motion and rescaled random walk.}
  \label{fig:overall}
\end{figure}
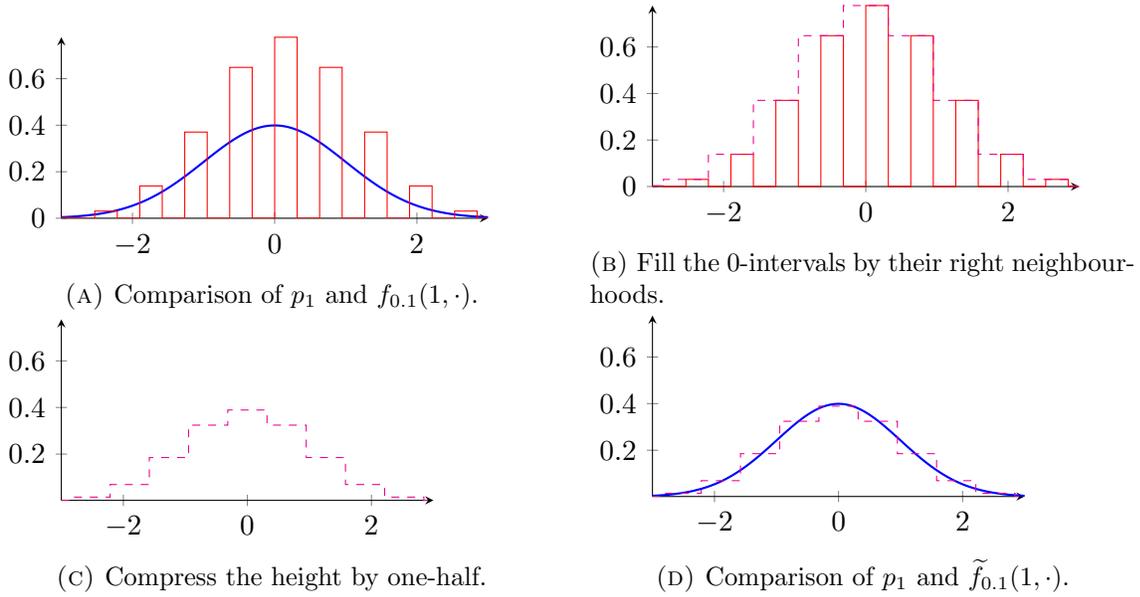

\subsection{Discrete Feynman-Kac approximation}\label{ss_main}

Let $u = \{u (t, x) \colon (t, x) \in \bR_{\geq 0} \times \bR\}$ be the solution to \eqref{pam} with a flat initial condition $u(0, \cdot) \equiv 1$. Then, we can formally write the Feynman-Kac representation for $u$ as follows
\begin{align}\label{eq_fk}
 u(t, x) = 
 \bE^B \bigg[\exp \Big(\int_0^t \int_{\bR} \delta (B_{t-s} + x - y) W(ds, dy) - \frac{1}{2}\| \delta (B_{t - \cdot} + x - *)\|_{\cH \otimes \cH_*}^2\Big)\bigg], 
\end{align}
where, $B$ is a Brownian motion independent of $W$, $\delta$ denotes the Dirac delta function at the origin, and $\cH$ and $\cH_*$ are Hilbert spaces corresponding to the time and space variables of $W$, respectively, which will be defined in Subsection \ref{ss_malliavin} below. For instance, when the driving noise is space-time white, 
\[
\cH \otimes \cH_* = L^2(\bR_{\geq 0}) \otimes L^2(\bR) = L^2(\bR_{\geq 0} \times \bR). 
\]

In the next step, we approximate \eqref{eq_fk} by substituting the Brownian motion with a simple random walk. 
Let $S = \{S_k \colon k = 0, 1, \dots\}$ be a simple random walk, namely, $S_0 = 0$, and $\{Z_k \coloneqq S_k - S_{k - 1}, k = 1, 2, \dots\}$ forms an i.i.d. sequence of Rademacher random variables with $\bP (Z_k = - 1) = \bP (Z_k = 1) = \frac{1}{2}$. 
For any $h > 0$, let 
\begin{align}\label{def_wh}
 W_h (m, n) \coloneqq 
 \frac{1}{2\sqrt{h}}\int_0^{\infty} \int_{\bR} \cT_h^{(m)} (t) \cS_h^{(n)} (x) W(dt, dx), 
\end{align}
for all $(m, n ) \in \bZ_{> 0} \times \bZ$, where $\cT_h^{(m)}$ and $\cS_h^{(n)}$ are indicator functions on intervals
$[(m - 1)h, mh )$ and $[2n \sqrt{h}, 2 (n + 1) \sqrt{h} )$, respectively. 
Then, 
$W_h = \{W_h (m, n) \colon (m, n) \in \bZ_{> 0} \times \bZ\}$ 
is a sequence of centred Gaussian random variables with covariance
\begin{align} \label{eq_cov-w_h}
 \bE \big[W_h (m_1, n_1) W_h (m_2, n_2)\big] = & (4h)^{-1} \big\langle \cT_h^{(m_1)}, \cT_h^{(m_2)} \big\rangle_{\cH} \big\langle \cS_h^{(n_1)}, \cS_h^{(n_2)} \big\rangle_{\cH_*}. 
\end{align}

Let $u_h$ be a random field index by $\bZ_{\geq 0} \times \bZ$, and for any $(m, n) \in \bZ_{\geq 0}\times \bZ$, 
\begin{align}\label{fk_h_X}
 u_h (m, n) \coloneqq & \bE^{S} \bigg[\wick \bigg(\sum_{i = 1}^{m} W_h \Big(i, \Big\lfloor\frac{S_{m + 1 - i} }{2} \Big\rfloor + n \Big) \bigg) \bigg], 
\end{align}
where $\wick$ is a function on the space of centred Gaussian random variables that are (random) functionals of $W$, given by
\begin{align}\label{def_wick}
 \wick (Z) \coloneqq \exp \Big(Z - \frac{1}{2}{\rm Var}^W(Z) \Big),
\end{align}
where ${\rm Var}^W$ denotes the variance with respect to $W$.

\begin{remark}
For all $m, n \in \bZ$, define
\begin{align}\label{def_gmm-t}
 \Gamma (m) \coloneqq \big\langle \1_{[|m|, |m| + 1]}, \1_{[0, 1]} \big\rangle_{\cH} = \begin{dcases}
 \frac{1}{2} \big(|m + 1|^{2H} + |m - 1|^{2H} - 2 |m|^{2H}\big), & H > \frac{1}{2}, \\
 \1_{\{0\}} (m), & H = \frac{1}{2};
 \end{dcases}
\end{align}
and
\begin{align}\label{def_gmm*-s}
 \Gamma_* (n) \coloneqq \big\langle \1_{[|n|, |n| + 1]}, \1_{[0, 1]} \big\rangle_{\cH_*} = \begin{dcases}
 \frac{1}{2} \big(|n + 1|^{2H_*} + |n - 1|^{2H_*} - 2 |n|^{2H_*}\big), & H_* > \frac{1}{2}, \\
 \1_{\{0\}} (n), & H_* = \frac{1}{2};
 \end{dcases}
\end{align}
see Subsection \ref{ss_malliavin} for a detailed discussion of the spaces $\cH$ and $\cH_*$. 
We can calculate its covariance as expressed in \eqref{eq_cov-w_h} by using the space-time homogeneity of the noise, and obtain
\begin{align}\label{eq_cov-wh}
 {\rm Cov} \big[W_h (m_1, n_1) W_h (m_2, n_2)\big] = & (4h)^{-1} H(2H - 1)   \int_0^{h} dt \int_{|m_2 - m_1|h}^{(|m_2 - m_1| + 1)h} ds |t - s|^{2H - 2} \nonumber\\
 & \times H_* (2H_* - 1) \int_0^{2\sqrt{h}} dx \int_{2|n_2 - n_1|\sqrt{h}}^{2(|n_2 - n_1| + 1)\sqrt{h}} dy |x - y|^{2H_* - 2} \nonumber\\
 = & 2^{4H_* - 2} h^{2H + H_* - 1} \Gamma (m_2 - m_1) \Gamma_* (n_2 - n_1). 
\end{align}
We note that in formula \eqref{eq_cov-wh},  the assumptions  $H > \frac{1}{2}$ and $H_* > \frac{1}{2}$ are made implicitly. However, these conditions are not essential and can be relaxed without difficulty.
Hence,
	\begin{align*}
	{\rm Var}^W & \bigg(\sum_{i = 1}^{m} W_h \Big(i, \Big\lfloor\frac{S_{m + 1 - i} }{2} \Big\rfloor + n \Big) \bigg) =  \sum_{i = 1}^m \sum_{j = 1}^m {\rm Cov}^W  \bigg( W_h \Big(i, \Big\lfloor\frac{S_{m + 1 - i} }{2} \Big\rfloor + n \Big), W_h \Big(i, \Big\lfloor\frac{S_{m + 1 - j} }{2} \Big\rfloor + n \Big) \bigg) \\ 
	= & 2^{4 H_* - 2} h^{2H + H_* - 1} \sum_{i = 1}^m \sum_{j = 1}^m \frac{1}{2} \big(|i - j + 1|^{2H} + |i - j - 1|^{2H} - 2 |i - j|^{2H}\big) \\ 
	& \qquad \times \frac{1}{2} \bigg(\Big|\Big\lfloor\frac{S_{m + 1 - i} }{2} \Big\rfloor  - \Big\lfloor\frac{S_{m + 1 - j} }{2} \Big\rfloor  + 1\Big|^{2H_*} + \Big|\Big\lfloor\frac{S_{m + 1 - i} }{2} \Big\rfloor  - \Big\lfloor\frac{S_{m + 1 - j} }{2} \Big\rfloor  - 1\Big|^{2H_*} \\ 
	& \qquad \qquad  - 2 \Big|\Big\lfloor\frac{S_{m + 1 - i} }{2} \Big\rfloor  - \Big\lfloor\frac{S_{m + 1 - j} }{2} \Big\rfloor \Big|^{2H} \bigg) ;
	\end{align*}
This observation, together with the definition \eqref{def_wick} of Wick renormalised exponential, indicates that for any realisation of $W$,  the quantity in  \eqref{fk_h_X} can be explicitly computed and yields a positive value. 
\end{remark}

The following theorem presents our first main result concerning the convergence rate of the discrete Feynman-Kac approximation to \eqref{pam} with a flat initial condition. The second main result, addressing the case of \eqref{pam} with a delta initial condition, is postponed to Theorem \ref{thm_rate-delta} in Section \ref{sec_delta}. 
\begin{theorem}\label{thm_rate-frac}
 Let $u$ be the solution to \eqref{pam} with a flat initial condition, namely, $u(0, \cdot) \equiv 1$, driven by a fractional Brownian sheet with Hurst parameters $H \geq \frac{1}{2}$ and $H_* \geq \frac{1}{2}$ in time and space, respectively. Let $u_h = \{u_h(m, n) \coloneqq (m, n) \in \bZ_{\geq 0}\times \bZ\}$ be defined in \eqref{fk_h_X} with some $h \in (0, 1)$. Then, for any $p \geq 2$, $(t, x) \in \bR_{\geq 0} \times \bR$ and $\epsilon \in (0, (2H + H_* - 1) \wedge 1 )$, 
 \begin{align}\label{ieq_rate-frac}
  \big\|u_h \big( \lfloor t/h \rfloor ,  \big\lfloor x/\sqrt{4h} \big\rfloor\big) - u(t, x) \big\|_{L^p (\Omega)}^2 \lesssim h^{(2 H + H_* - 1)\wedge 1 - \epsilon}, 
 \end{align}
 where, for any nonnegative functions $A$ and $B$ defined on the same domain $\rm Dom$, $A(x)\lesssim B(x)$ indicates that there exists an (implicit) universal constant $c > 0$, such that $A(x) \leq c B(x)$ for all $x \in {\rm Dom}$, and the implicit constant in \eqref{ieq_rate-frac} depends on $p$, $t$ and $\epsilon$. 
\end{theorem}

The convergence rate established in Theorem \ref{thm_rate-frac}, and also in Theorem \ref{thm_rate-delta} for a delta initial condition, is `almost' optimal. Specifically, for a fully discrete scheme, one cannot expect the convergence rate to exceed the exponent of the H\"{o}lder continuity. As stated in Theorem \ref{thm-holder} below, the exponent in \eqref{ieq_rate-frac} matches the time exponent in \eqref{eq_holder}. Notably, the small correction order $\epsilon$ in \eqref{eq_holder} could potentially be eliminated through a more detailed analysis, thereby achieving matching upper and lower bounds. This reinforces that our convergence rate result is `almost' optimal, up to an arbitrarily small correction $\epsilon > 0$. We hope this correction term $\epsilon$ can be eliminated in future studies, achieving the optimal order as proved for lattice approximation with space-time white noise in \cite[Theorem 3.1]{pa-99-gyongy}. 

Another remark is on how the discretisation depends on the terminal time $t$. We take the time step $h$ independent of $t$. Therefore, the number of steps is $N = \lfloor t/ h \rfloor$, which grows linearly with $t$. The associated simple random walk $\{S_1, \dots, S_{N}\}$ has $2^{N}$ possible paths, hence the computation is exponential in $t$. Moreover, the implicit constant in \eqref{ieq_rate-frac} should be comparable to  $\|u (t, x)\|_p^2$, which grows exponentially in $t$ as well. For example, in the case of space-time white noise the growth is of order $\exp(\frac{1}{24}p(p^2 - 1) t)$; see \cite[Formula (1.7)]{aihp-15-chen}.

A related work \cite{spde-18-joseph} investigates a discrete stochastic heat equation with a Lipschitz continuous diffusion coefficient, driven by i.i.d. centred random variables. It is shown that, under appropriate spatial scaling, the solution to the discrete equation converges in distribution to that of its continuum counterpart. A distinguishing feature of Joseph’s result, compared to other works on directed polymers in random environments, is the derivation of moment convergence. However, this does not imply the strong convergence established in the present paper.

After completing the first version of the manuscript, we became aware of several related works that studied positivity-preserving splitting schemes for stochastic heat equations. In particular, \cite{puz-24-brehier-cohen-ulander} considers equations driven by Brownian motion, while \cite{esimmna-24-brehier-cohen-ulander} addresses the case of space-time white noise. Both works focus on bounded spatial domains and employ the Lie-Trotter splitting method to construct and analyse their schemes. In contrast, our work addresses \eqref{pam} driven by noise on the entire spatial domain, which may also be coloured in space and/or time.

\subsection{Prospects for future work}

As an initial step in this direction, this paper focuses on the one-dimensional equation and restricts the driving noise to fractional Brownian sheets that are white or more regular than white in both time and space. Consequently, several potential avenues exist for extending our results. For instance, it would be interesting to explore the case of multiple spatial dimensions, rough noises as studied in \cite{abel-18-hu-huang-le-nualart-tindel, aihp-19-chen}, general Gaussian noises satisfying the reinforced Dalang condition introduced in \cite{ams-00-sanz-sarra}, and time-derivative noise as in \cite{ap-12-hu-lu-nualart}. 

Additionally, as previously discussed, approximating the density of simple random walks requires a slight modification to the standard rescaling $\sqrt{h} S_{\lfloor t/h\rfloor}$ to address issues arising from the alternating state space in odd and even steps. Another potentially applicable approach might involve replacing the simple random walk with other types of random walks, such as 
 $X_n = Z_1 + \dots + Z_n$, where $Z_i$'s are i.i.d. random variables taking values in $\{-1, 0, 1\}$ with probabilities $\bP(Z_i = - 1) = \bP(Z_i = 1) = \alpha$ and $\bP(Z_i = 0) = 1 - 2\alpha$, for some $\alpha \in (0, 1/2)$. It would be valuable to know whether this adjustment suffices to achieve strong convergence and whether the parameter $\alpha$ influences the convergence rate. 

Furthermore, as highlighted in the abstract and Section \ref{sec_dpre}, our theorems provide a quantitative framework for analysing the convergence of the partition function of directed polymers in Gaussian environments. A natural next step would be extending these results to other types of random environments. For general random environments, one possible approach is coupling the environments as Gaussian functionals, allowing the strategies developed in this paper to be adapted to the new setting. Alternatively, one could pursue a quantitative analysis of the convergence of polynomial chaos to Wiener chaos. The Lindeberg principle could be a useful tool for addressing this question. It provides a quantitative convergence of polynomial chaos to Wiener chaos (cf. \cite[Theorem 2.6]{jems-16-caravenna-sun-zygouras} and \cite[Theorem 3.18]{am-10-mossel-odonnell-krzysztof}), expressed in terms of the following metric 
\[
 d (\mu, \nu) \coloneqq \sup_{\|\phi\|_{3, \infty} \leq 1} \Big|\int_{\mathbb{R}} \phi(x) \mu (d x) - \int_{\mathbb{R}} \phi(x) \nu (d x) \Big|, 
\]
where $\|\phi\|_{3, \infty} \coloneqq \|\phi\|_{\infty} + \|\phi'\|_{\infty} + \|\phi''\|_{\infty} + \|\phi'''\|_{\infty} $. One challenge in applying the Lindeberg principle is estimating the `influence' of a single environment random variable on the partition function, which is particularly difficult for space-time dependent environments. We hope these challenges can be addressed in future work, potentially opening new horizons for research in this field.

\subsection{Organisation of the paper} 
This paper is organised as follows. In Section \ref{sec_pre}, we provide some necessary preliminaries. Section \ref{sec_flat} focuses on \eqref{pam} with a flat initial condition, and includes the proof of Theorem \ref{thm_rate-frac}. The \eqref{pam} with a delta initial condition is explored in Section \ref{sec_delta}, where a parallel result, Theorem \ref{thm_rate-delta}, to Theorem \ref{thm_rate-frac} is established. Section \ref{sec_dpre} contains a reinterpretation of our results in the context of directed polymers in random environments. 

\section{Preliminaries}\label{sec_pre}

\subsection{Notation}\label{ssec_nota}
Let $h > 0$ and $k \in \bZ_{>0}$ be some generic constants, and let $(t, x) \in \bR_{> 0} \times \bR$ be some generic parameters. 
Denote by $G \colon \bZ_{\geq 0} \times \bZ \to \bR$ the probability mass function of the simple random walk $S$, namely, 
\begin{align}\label{def_Gh}
 G(m, n) \coloneqq \bP \big(S_m = n\big), \quad (m, n) \in \bZ_{\geq 0} \times \bZ. 
\end{align}
For any $y \in \bR$, we write
\begin{align}\label{def_t_h-x_h}
 \lfloor s \rfloor_h \coloneqq \lfloor s/h \rfloor .
 \end{align}
 Let $(\bfs_k, \bfy_k) = (s_1, \dots, s_k, y_1, \dots, y_k)\in [0, t]^{k} \times \bR^k$. We introduce a permutation $\sigma$ on $\{1, \dots, k\}$ such that 
 \[
 0 \leq s_{\sigma (1)} \leq s_{\sigma(2)} \leq \dots \leq s_{\sigma (k)} \leq t, 
 \] 
 and make use of notation 
 \[
 [0, t]_<^k \coloneqq \big\{\bfs_k \in [0, t]^k, s_1 < s_2 <\dots < s_k <t \big\}. 
 \]
Additionally, 
$(\ks_0, \ks_1, \dots, \ks_k) \in \bZ_{\geq 0}^{k + 1}$ and $(\ky_0, \ky_1, \dots, \ky_k) \in \bZ^{k + 1}$ are defined as follows, 
\begin{align}\label{def_ks}
 \ks_i \coloneqq \lfloor s_{\sigma(i + 1)}\rfloor_h - \lfloor s_{\sigma(i)}\rfloor_h , 
  \end{align}
 and 
  \begin{align}\label{def_ky}
  \ky_i \coloneqq &
  2 \big[\lfloor y_{\sigma(i + 1)}\rfloor_{2\sqrt{h}} - \lfloor y_{\sigma(i)}\rfloor_{2\sqrt{h}}\big] + \tau\big(\lfloor t \rfloor_h - \lfloor s_{\sigma(i + 1)} \rfloor_h \big) - \tau \big(\lfloor t \rfloor_h - \lfloor s_{\sigma(i)} \rfloor_h\big), 
  \end{align}
  for all $i = 0, \dots, k$ with convention $(s_{\sigma(0)}, y_{\sigma(0)}) \coloneqq (0, 0)$ and $(s_{\sigma(k + 1)}, y_{\sigma(k + 1)}) \coloneqq (t, x)$, where $\tau \colon \bZ \to \{0, 1\}$ is the alternating indicator given by
 \begin{align}\label{def_tau-t}
\tau(n) = \begin{dcases}
 0, & n \text{ is even}, \\
1, & n \text{ is odd}. 
\end{dcases}
\end{align}
Here, the alternating indicator $\tau$ is introduced to ensure that $\ky_i$ in \eqref{def_ky} lies in the state space of $S_{\ks_i}$. Without $\tau$, $\ky_i$ would take values only in the set of even integers, which would fall outside the state space whenever $\ks_i$ is an odd integer.

\subsection{Multiple integrals and chaos expansion}\label{ss_malliavin}
Let $W$ be a fractional Brownian sheet with Hurst parameters $H \geq \frac{1}{2}$ and $H_* \geq \frac{1}{2}$ in time and space, respectively. Then, it can be also considered as an isonormal Gaussian process on a Hilbert space $\kH = \cH \otimes \cH_*$, where $\cH$ is the completion of the space of test functions on $\bR_{\geq 0}$ with respect to the inner product
\begin{align*}
 \langle \phi, \psi \rangle_{\cH} = \int_{\bR_{\geq 0}^2} \gamma (t - s) \phi (s) \psi (t) ds dt; 
\end{align*}
and similarly $\cH_*$ is the completion of the space of test functions on $\bR$ with respect to the inner product
\begin{align*}
 \langle \phi, \psi \rangle_{\cH_*} = \int_{\bR^2} \gamma_* (x - y) \phi (x) \psi (y) dx dy;
\end{align*}
where, with $C_H \coloneqq H (2H-1)$ and $C_{H_*} \coloneqq H_* (2H_*-1)$, 
\begin{align*}
 \gamma (t) \coloneqq \begin{dcases}
  C_H |t|^{2H - 2} , & H > \frac{1}{2}, \\
  \delta(t), & H = \frac{1}{2};
 \end{dcases}
 \quad \text{and} \quad \gamma_* (x) \coloneqq \begin{dcases}
  C_{H_*} |x|^{2H_* - 2} , & H_* > \frac{1}{2}, \\
  \delta(x), & H_* = \frac{1}{2};
 \end{dcases}
\end{align*}
for all nonzero real numbers $t$ and $x$. 
Then, for any $\Phi, \Psi \in \kH$, 
\begin{align*}
 \bE [W(\Phi) W(\Psi)] = & \bE \bigg[ \int_{\bR_{\geq 0}} \int_{\bR} \Phi (t, x) W(dt, dx) \times \int_{\bR_{\geq 0}} \int_{\bR} \Psi (t, x) W(dt, dx) \bigg] 
 = \langle \Phi, \Psi\rangle_{\kH}. 
\end{align*}

Let $\{H_{k} \colon k \in \bZ_{
\geq 0} \}$ denote the (probabilist's) Hermite polynomials, namely, 
\begin{align}\label{def_hermite}
 H_k (x) =
\frac{(-1)^k}{k!} e^{\frac{x^2}{2}} \frac{d^k}{d x^k} e^{-\frac{x^2}{2}}. 
\end{align}
Let $\bfH_0 = \bR$; and for any positive integer $k$, let $\bfH_k$ be the closed linear subspace spanned by random variables 
$\{H_{k} (W(\Phi)) \colon \Phi \in \kH, \|\Phi\|_{\kH} = 1\}$. 
Then, $\bfH_k$ is called the $k$-th chaos. Denote by $\kH^{\odot k}$ the subspace of $\kH^{\otimes k}$ consisting of all symmetric functions. It turns out that the map $I_k$ acting on $\kH^{\odot k}$ characterised by 
\[
	I_k (\Phi^{\otimes k}) = \|\Phi\|_{\kH}^k I_k \bigg(\Big(\frac{\Phi}{\|\Phi\|_{\kH}}\Big)^{\otimes k}\bigg) \coloneqq \|\Phi\|_{\kH}^k H_k \Big(W \Big(\frac{\Phi}{\|\Phi\|_{\kH}}\Big)\Big)
\]
is a linear isometry between $\kH^{\odot k}$ and $\bfH_k$ with the modified norm $\sqrt{k!} \|\cdot\|_{\kH^{\otimes k}}$. In view of \cite[Proposition 3.5]{ws-17-hu}, the Wick renormalised exponential \eqref{def_wick} can be equivalently expressed as a series involving Hermite polynomials. Specifically, ignoring the trivial case $Z \equiv 0$, 
\begin{align}\label{def_wick-herm}
	\wick (Z) = 1 + \sum_{n = 1}^{\infty} \frac{t^n}{n!} H_n \Big(\frac{Z}{t}\Big) , \quad \text{where } t \coloneqq \sqrt{{\rm Var} (Z)}.
\end{align}

Moreover, for any $W$-measurable square integrable random variable $X$, it can be decomposed as an infinite sum of uncorrelated random variables $I_k (f_k) \in \bfH_k$, namely, 
\[
X = \bE[X] + \sum_{k = 1}^{\infty} I_k(f_k), \quad \text{where } f_k \in \kH^{\odot k} \text{ is determined by } X;
\]
cf. \cite[Theorem 1.1.1]{springer-06-nualart}. 

The Skorokhod integral extends the It\^{o} and the Dalang-Walsh integrals to cases where the integrands are non-adapted processes, or/and the driving noises are non-martingale measures. Let $v = \{v (t, x) \colon (t, x) \in \bR_{\geq 0} \times \bR\}$ be a $W$-measurable random field. Suppose that for each $(t, x)$, $v(t, x)$ is square integrable. Then, it has a chaos expansion, 
\begin{align*}
 v(t, x) = \sum_{k = 0}^{\infty} I_{k} (f_k(t, x, \cdot)). 
\end{align*}
Particularly, $v$ is called Skorokhod integrable with respect to $W$ (cf. \cite[Section 1.3.2]{springer-06-nualart}), if both of the following statements are satisfied. 
\begin{itemize}
 \item For each $k \geq 0$, $\widetilde{f}_k \in \kH^{\odot (k + 1)}$, where $\widetilde{f}_k \colon \bR_{\geq 0}^{k + 1} \times \bR^{k + 1} \to \bR$, is the symmetrisation of $f_k$, given by
 \begin{align*} 
  \widetilde{f}_k(z_1, \dots, z_{k + 1}) = \frac{1}{k + 1} \sum_{i = 1}^{k + 1} f_k \big(z_i, z_2, z_3, \dots, z_{i - 1}, z_1, z_{i + 1}, \dots, z_{k + 1}\big), 
 \end{align*}
with $z_i \coloneqq (s_i, x_i)$, $i = 1, \dots, k + 1$. 

\item The following inequality holds
\[
 \sum_{k = 0}^{\infty} (k + 1)! \big\|\widetilde{f}_k\big\|_{\kH^{\otimes (k + 1)}}^2 < \infty. 
\]
\end{itemize}
Moreover, 
the Skorokhod integral of $v$ has the chaos expansion as follows, 
\begin{align*}
 \delta(v) = \int_{\bR_{\geq 0}} \int_{\bR} v(t, x) W(d t, d x) = \sum_{k = 0}^{\infty} I_{k + 1} \big(\widetilde{f}_k \big) . 
\end{align*}
For a more detailed account of the Skorokhod integrals, we refer the readers to the monographs such as \cite{ws-17-hu, springer-06-nualart}. 

 \subsection{Results on \eqref{pam}}\label{ss_pre-pam}

 A random field $u = \{u(t, x) \colon (t, x) \in \bR_{> 0} \times \bR\}$ is called a solution to \eqref{pam} with initial condition $u(0, \cdot) = \mu$ being a signed Radon measure, if the following mild formulation holds
 \begin{align*} 
  u (t, x) = \int_{\bR} p_{t} (x- y) \mu (dy) + \int_0^t \int_{\bR} p_{t-s} (x-y) u (s, y) W(d s, d y) , 
 \end{align*}
 for all $(t, x) \in \bR_{> 0} \times \bR$, where $p_t(x) \coloneqq \frac{1}{\sqrt{2\pi t}} e^{-\frac{x^2}{2t}}$, denotes the heat kernel and the stochastic integral is the Skorokhod integral. For the sake of conciseness,  we only study \eqref{pam} with either a flat or a delta initial conditions in this paper, and refer the reader to \cite{ap-15-chen-dalang} concerning the existence, uniqueness and certain properties of solutions to SPDEs with more general initial conditions. Below, we present several existing results for \eqref{pam}.

The next theorem corresponds to the existence and uniqueness of solutions to \eqref{pam}, and the chaos expansion for the solution. 

 \begin{theorem}\label{thm_eu}
  Let $W$ be a fractional Brownian sheet with Hurst parameters $H \geq \frac{1}{2}$ and $H_* \geq \frac{1}{2}$ in time and space, respectively. Then, the following properties hold. 
  \begin{enumerate}
   \item[1.] With a flat initial condition, namely, $u(0, \cdot) \equiv 1$,  equation \eqref{pam} has a unique solution, which can be expressed as following the chaos expansion 
   \begin{align*} 
 u(t, x) = 1 + \sum_{k = 1}^{\infty} \frac{1}{k!} I_k(f_{k}), 
\end{align*}
where $I_k$ denotes the multiple Wiener-It\^{o} integral and $f_{k} \colon [0, t]^k \times \bR^k \to \bR$ given by
\begin{align}\label{def_fk}
f_k (\bfs_k, \bfy_k) = f_{k}^{t, x} (\bfs_k, \bfy_k) \coloneqq 
 & \prod_{i = 1}^k p_{s_{\sigma (i + 1)} - s_{\sigma (i)}} (y_{\sigma(i + 1)} - y_{\sigma(i)}), 
\end{align}
$\sigma$ is the permutation on $\{1, \dots, k\}$ defined in Subsection \ref{ssec_nota} with convention  $(s_{\sigma (k + 1)}, y_{\sigma (k + 1)}) \break = (t, x)$. 
Moreover, with a universal constant $C > 0$, for every $k \in \bZ_{> 0}$, 
\begin{align}\label{eq_est-chaos-flat} 
 \|f_k(\bfs_k, \cdot)\|_{\cH_*}^2 \leq 
  C^k \prod_{i = 1}^k (s_{\sigma(i + 1)} - s_{\sigma(i)})^{H_* - 1} . 
\end{align}
\item[2.] With a delta initial condition at $z \in \mathbb{R}^d$, namely, $u (0, \cdot) = \delta(z - \cdot)$, equation \eqref{pam} has a unique solution, which can be expressed as following the chaos expansion 
   \begin{align*} 
 u(t, x) = p_{t}(x - z) + \sum_{k = 1}^{\infty} \frac{1}{k!} I_k(f_{k}), 
\end{align*}
  where
\begin{align}\label{def_fk-delta}
 f_k (\bfs_k, \bfy_k) = f_{k}^{t, x} (\bfs_k, \bfy_k) \coloneqq 
  & \prod_{i = 0}^k p_{s_{\sigma (i + 1)} - s_{\sigma (i)}} (y_{\sigma(i + 1)} - y_{\sigma(i)}), 
 \end{align}
with convention $(s_{\sigma (0)}, y_{\sigma (0)}) = (0, z)$ and $(s_{\sigma (k + 1)}, y_{\sigma (k + 1)}) = (t, x)$. Moreover, with a universal constant $C > 0$,  for every $k \in \bZ_{> 0}$, 
\begin{align}\label{eq_est-chaos-delta}
 \|f_k(\bfs_k, \cdot)\|_{\cH_*}^2 \leq 
C^k t^{-H_*} \prod_{i = 0}^k (s_{\sigma(i + 1)} - s_{\sigma(i)})^{H_* - 1}. 
\end{align}
\end{enumerate}
\end{theorem}

The next theorem addresses the H\"{o}lder continuity of the solution to \eqref{pam}. 

\begin{theorem}\label{thm-holder}
  Let $u$ be the solution to \eqref{pam} satisfying the same condition as in Theorem \ref{thm_eu}. Then, for any $s, t \in (0, T]$ and $x, y \in \bR$, it holds that
 \begin{align}\label{eq_holder}
  \bE \big[ |u(t, x) - u(s, y)|^2 \big] \lesssim |t - s|^{2H + H_* - 1 - \epsilon} + |x - y|^{(4H + 2H_* - 2) \wedge 2 - \epsilon}, 
 \end{align}
 for all $\epsilon \in (0, 2H + H_* - 1)$, where the implicit constant depending on $\epsilon$ and $T$, and also $s \wedge t$ for a delta initial condition. 
\end{theorem}

Theorems \ref{thm_eu} and \ref{thm-holder} are cited from \cite[Lemmas 3.2 and 3.5]{jtp-18-balan-chen}, \cite[Theorem 2.2 and the display after Formula (2.3)]{amsc-19-balan-song-quer_sardanyons}, and \cite[Theorem 4.1]{bej-24-guo-song-song}. The H\"{o}lder continuity was also investigated in \cite{amsc-19-balan-song-quer_sardanyons}, though with a non-optimal exponent. In contrast, Theorem \ref{thm-holder}, cited from \cite{bej-24-guo-song-song}, provides an estimate that is `almost' optimal. A lower bound for the moments of the projection of $u(t, x) - u(s, y)$ onto the first chaos can be established as $|t - s|^{2H + H_* - 1} + |x - y|^{(4 H + 2H_* - 2)\wedge 2}$; cf. \cite[Remark 3.5]{saa-20-herrell-song-wu-xiao}\footnote{When $4H + 2H_* - 2 = 2$,  a logarithmic correction may appear, resulting in the lower bound involves a quantity like $|x - y|^2 (1 + |\log (|x - y|)|)$. }. The correction $\epsilon$ in \eqref{eq_holder} could potentially be eliminated with a more careful analysis. However, since our primary interest lies in the rate of convergence of approximate solutions, where this $\epsilon$ manifests as an approximation error, the inequality \eqref{eq_holder} is already sufficiently sharp for our purposes. Consequently, we do not pursue further optimisation of the H\"{o}lder exponent. 

\subsection{Technical lemmas}

In this subsection, we present several technical lemmas that will be utilised in the proofs of our main results in Sections \ref{sec_flat} and \ref{sec_delta}.

\begin{lemma}[Hardy-Littlewood-Sobolev inequality, cf. {\cite[Theorem 1.1]{spl-01-memin-mishura-valkeila}}] \label{lmm_hls}
Let $f \in \cH$ with $H > \frac{1}{2}$ (see Subsection \ref{ss_malliavin}). Furthermore, suppose that $f$ is supported in a compact interval $[0,T]$. Then, 
\begin{align*}
 \|f\|_{\cH} \lesssim \|f\|_{L^{1/H}}. 
\end{align*}
\end{lemma}

\begin{lemma}
 [Tail probability for normals, {cf. \cite[Formula (10)]{ams-41-gordon}}]\label{tail-normal}
 Let $X$ be a standard normal random variable, then for all $x > 0$, 
 \[
  \frac{x}{x^2 + 1} p_1(x) \leq \bP(X \geq x) \leq \frac{1}{x} p_1 (x), 
 \]
 where $p_1 (x) = \frac{1}{\sqrt{2 \pi}} e^{-\frac{x^2}{2}}$ is the heat kernel at $t = 1$. 
\end{lemma}

\begin{lemma}
 [Bernstein's inequality, cf. {\cite[Page 34]{jasa-62-bennett}}]\label{tail-binom}
 Let $S_n$ be the summation of $n$ i. i. d Rademacher random variables with equal probability. Then, for any $t > 0$, 
 \[
  \bP(S_n \geq t \sqrt{n}) < \exp \Big(- \frac{t^2}{2 + \frac{2 t}{3 n}}\Big). 
 \]
\end{lemma}

\begin{lemma}[Local limit theorem, cf. {\cite[Theorem 6, p. 197]{springer-75-petrov}}]\label{lmm_g-p}
 Let $G$ be given by \eqref{def_Gh}, and let $p$ denote the heat kernel. Then, for any integers $m$ and $ n$, such that $m > 0$ and $m + n$ is even, and real numbers $z \in [n - 2, n + 2]$ the following inequality holds
 \begin{align}\label{ieq_g-p}
  \big| G (m, n) - 2 p_m (z) \big| \lesssim m^{- 1} ,
 \end{align}
 where we recall that $p_m (z) = \frac{1}{\sqrt{2 \pi m}} e^{-\frac{z^2}{2 m}}$ denotes the heat kernel at time $t = m$.
\end{lemma}

\begin{remark}
	The original statement of \cite[Theorem 6, p. 197]{springer-75-petrov} assumes that the maximal span of $Z_k = S_k - S_{k - 1}$ is $1$, leading to a formulation that differs from \eqref{ieq_g-p}. In our setting, however, $Z_k \in \{- 1, 1\}$ is a Rademacher random variable, whose  the maximal span is $2$. Therefore, we rescale the random walk by a factor of $1/2$, which yields
	\[
		 \bigg| \frac{1}{2} \sqrt{m} \bP \Big(\frac{S_m}{2} = \frac{n}{2} \Big) - \frac{1}{\sqrt{2 \pi}} \exp \bigg(- \frac{1}{2} \Big(\frac{n/2}{\sqrt{m} / 2}\Big)^2\bigg) \bigg| \lesssim \frac{1}{\sqrt{m}}.
	\]
	This estimate, together with a Taylor expansion of the exponential function, implies inequality  \eqref{ieq_g-p}.
\end{remark}

The next corollary is a direct result of Lemma \ref{lmm_g-p}, the scaling property of the heat kernel, namely, for all $h > 0$ and $(t, x) \in \bR_{> 0} \times \bR$, 
\begin{align*}
 p_t (x) = h^{-\frac{1}{2}} p_{t/h} \big(x/\sqrt{h}\big). 
\end{align*}

\begin{corollary}\label{coro_g-p}
 Assume conditions in Lemma \ref{lmm_g-p}. For any $t \geq h > 0$, $x \in \bR$, recalling notation \eqref{def_t_h-x_h}, with a universal implicit constant, we deduce that
 \begin{align*} 
  \Big| \frac{1}{2\sqrt{h}}G \big(\lfloor t \rfloor_h, 2\lfloor x \rfloor_{2\sqrt{h}} \pm \tau(\lfloor t \rfloor_h) \big) - p_t (x) \Big| \lesssim \frac{\sqrt{h}}{t}, 
 \end{align*}
where $\tau$ is given by \eqref{def_tau-t}, and `$\pm$' indicates that the inequality holds with either `$+$' or `$-$'. 
\end{corollary}

\begin{lemma}\label{lmm_g-h1}
   Let $G$ be defined in \eqref{def_Gh}. Then, for any $h > 0$,  $t > 0$, and $x \in \bR$, 
\begin{align*}
  \Big\| \frac{1}{2\sqrt{h}} G \big(\lfloor t \rfloor_{h},  2 (\lfloor x \rfloor_{2 \sqrt{h}} - \lfloor \cdot \rfloor_{2 \sqrt{h}}) \pm \tau(\lfloor t \rfloor_h) \big) \Big\|_{\cH_*}^2 \lesssim t^{H_* - 1}, 
\end{align*}
where $\tau$ is defined in \eqref{def_tau-t},  `$\pm$' indicates that the inequality holds with either `$+$' or `$-$', and the implicit constant depends only on $H_*$. 
\end{lemma}

\begin{proof}
Here we only provide the deduction assuming $H_* > \frac{1}{2}$, while the case when $H_* = \frac{1}{2}$ essentially follow the same idea. 
First, if $t < h$, then $\lfloor t \rfloor_{h} = 0$, and thus 
\begin{align*} 
 G \big(\lfloor t \rfloor_{h},  2 (\lfloor x \rfloor_{2 \sqrt{h}} - \lfloor \cdot \rfloor_{2 \sqrt{h}}) \pm \tau(\lfloor t \rfloor_h) \big) = & G \big(\lfloor t \rfloor_{h},  2 (\lfloor x \rfloor_{2 \sqrt{h}} - \lfloor  y \rfloor_{2 \sqrt{h}}) \big) \\ 
 = & \begin{dcases}
 1, & y \in \big[2\sqrt{h} \lfloor x \rfloor_{2 \sqrt{h}} \, , \, 2\sqrt{h} (\lfloor x \rfloor_{2 \sqrt{h}} + 1) \big), \\
0, & \text{otherwise}. 
 \end{dcases}
\end{align*}
It follows that 
\begin{align*}
 \Big\| \frac{1}{2\sqrt{h}} G \big(\lfloor t \rfloor_{h},  2 (\lfloor x \rfloor_{2 \sqrt{h}} - \lfloor \cdot \rfloor_{2 \sqrt{h}}) \big) \Big\|_{\cH_*}^2 = (4h)^{-1} \int_{0}^{2\sqrt{h}} dy \int_{0}^{2\sqrt{h}} dz |y - z|^{2H_* - 2} \lesssim h^{H_* - 1} \lesssim t^{H_* - 1}. 
\end{align*}
because $t < h$ and $H_* < 1$. 

Next, suppose $t \geq h$. Without loss of generality, assume $t/h$ is a positive even integer; the other cases follow similarly. 
It follows from Corollary \ref{coro_g-p}, Lemma \ref{lmm_inn_heat} below, and the fact that $G (m, \cdot)$ is supported on $\{0, \pm 1, \dots, \pm m\}$,  that
\begin{align*}
 \Big\| \frac{1}{2\sqrt{h}} G \big(\lfloor t \rfloor_{h},  2 (\lfloor x \rfloor_{2 \sqrt{h}} - \lfloor \cdot \rfloor_{2 \sqrt{h}})\big) \Big\|_{\cH_*}^2 \lesssim & h t^{-2} \int_{- \frac{3 t}{ \sqrt{h}} }^{\frac{3 t}{ \sqrt{h}}} \int_{- \frac{3 t}{ \sqrt{h}}}^{\frac{3 t}{ \sqrt{h}}} |y - z|^{2H_* - 2} d y dz + \|p_{t}(x - \cdot)\|_{\cH_*}^2\\
 \lesssim & h^{1 - H_*} t^{2 H_* - 2} + t^{H_* - 1} \lesssim t^{H_* - 1}, 
\end{align*}
because $t \geq h$ and $H_* < 1$. The proof of this lemma is complete. 
\end{proof}

This section concludes with the next lemmas on the heat kernel. 
\begin{lemma}\label{lmm_inn_heat}
 For any $t > 0$ and $x \in \bR$, it holds that
 \[
  \|p_{t}(x - \cdot)\|_{\cH_*}^2 = \|p_{t}\|_{\cH_*}^2 \lesssim t^{H_* - 1}. 
 \]
\end{lemma}
\begin{proof}
The proof of this lemma is elementary and can be established by e. g., using the Fourier representation of the inner product in $\cH_*$. We omit the details for brevity. 
\end{proof}

\begin{lemma}[{\cite[Lemma A. 4]{ap-15-chen-dalang}}] \label{lmm_heat-fac}
 For all $s, t > 0$ and $x, y \in \bR$, the next formula holds 
 \[
  p_s(x - y) p_t(y) = p_{\frac{st}{s + t}} \Big(\frac{ t}{s + t}x - y\Big) p_{s + t} (x). 
 \]
\end{lemma}

\section{\eqref{pam} with flat initial condition} \label{sec_flat}
In this section, we establish the rate of convergence in $L^2(\Omega)$ for the discrete Feynman-Kac approximation to \eqref{pam} with a flat initial condition. To facilitate clarity, we divide the analysis into two subsections. Subsection \ref{sec_white} focuses on the case of space-time white noise, while other cases are handled in Subsection \ref{sec_colour}. Although the proofs for both cases share a similar approach, we separate them to highlight the core concepts in the simpler space-time white noise case, without the additional technical complexities appearing with coloured noises. Moreover, the space-time white noise case serves as the most classical and significant example of \eqref{pam}, meriting a detailed and thorough treatment. 
  \subsection{White noise}\label{sec_white}
  In this subsection, we establish the convergence rate in $L^p (\Omega)$ for the approximation to \eqref{pam} driven by space-time white noise as in the next proposition. 

  \begin{proposition}\label{thm_rate-l2}
   Let $u$ be the solution to \eqref{pam} driven by space-time white noise, and let $u_h = \{u_h(m, n) \coloneqq (m, n) \in \bZ_{> 0}\times \bZ\}$ be defined in \eqref{fk_h_X} with some $h \in (0, 1)$. Then, for any $p \geq 2$, $(t, x) \in \bR_{\geq 0} \times \bR$, and $\epsilon \in (0, \frac{1}{2})$, 
   \begin{align*}
    \big\|u_h ( \lfloor t \rfloor_h ,  \lfloor x \rfloor_{2 \sqrt{h}}) - u(t, x) \big\|_{L^p(\Omega)}^2 \lesssim h^{\frac{1}{2} - \epsilon}, 
   \end{align*}
   where the implicit constant depends on $t$ and $\epsilon$. 
  \end{proposition}

Proposition \ref{thm_rate-l2} relies on Lemmas \ref{lmm-chaos-h}--\ref{prop_l2-f-g} below, with proofs deferred to the end of this subsection. 

 \begin{lemma}\label{lmm-chaos-h}
  Let $u_h = \{u_h (m, n) \colon (m, n) \in \bZ_{> 0} \times \bZ\}$ be given by \eqref{fk_h_X}, and let $g_{h, k} \colon [0, mh]^k \times \bR^k \to \bR$ be given by
  \begin{align}\label{def_gk}
   g_{h, k} (\bfs_k, \bfy_k)  =  g_{h, k}^{t,x} (\bfs_k, \bfy_k)  = & 2^{-k} h^{- \frac{k}{2}} \prod_{j = 1}^k G\big(\ks_j, \ky_j \big), 
  \end{align}
  where $G$,  $(\ks_1, \dots, \ks_k)$, and $(\ky_1, \dots, \ky_k)$ are defined in \eqref{def_Gh}, \eqref{def_ks} and \eqref{def_ky}, respectively, with the parameters $(t, x) \in \bR_{> 0} \times \bR$, such that $\lfloor t \rfloor_h = m$ and $\lfloor x \rfloor_{2 \sqrt{h}} = n$. Then, $u_h$ can be represented as the following chaos expansion, 
  \begin{align}\label{chaos-uh}
  u_h (m, n) = & 1 + \sum_{k = 1}^{\infty} \frac{1}{k!} I_k(g_{h, k}). 
  \end{align}  
 \end{lemma}

\begin{lemma}\label{lmm-l2-g-p}
 Fix $h > 0$. Let $p$ denote the heat kernel, and let $G$ be defined in \eqref{def_Gh}. Then, for any $(t, x) \in \bR_{> 0} \times\bR$, the following inequality hold with any $\epsilon \in (0, \frac{1}{2})$, 
  \begin{align}\label{eq_l2-g-p}
   \int_{\bR} \Big|\frac{1}{2\sqrt{h}} G \big(\lfloor t \rfloor_h, 2 (\lfloor x \rfloor_{2 \sqrt{h}} - \lfloor y \rfloor_{2 \sqrt{h}}) \pm \tau(\lfloor t \rfloor_h) \big) - p_{t} (x -y)\Big|^2 d y \lesssim h^{\frac{1}{2} - \epsilon} t^{-1 + \epsilon} , 
  \end{align}
where $\tau$ is given by \eqref{def_tau-t}, and `$\pm$' indicates the inequality holds with either `$+$' or `$-$'. 
 \end{lemma}

 \begin{lemma}\label{prop_l2-f-g}
  Fix $h > 0$, $k \in \bZ_{>0}$, and $(s_1, \dots, s_k) \in [0, t]_<^k$. Let $f_k$ and $g_{h, k}$ be given by \eqref{def_fk} and \eqref{def_gk}, respectively. Then, for any $\epsilon \in (0, 1/2)$, 
  \begin{align*}
   \int_{\bR^k} \big| f_k (\bfs_k, \bfy_k) - g_{h, k} (\bfs_k, \bfy_k)\big|^2 d \bfy_k \leq C^k h^{\frac{1}{2} - \epsilon} \sum_{j = 1}^k (s_{j + 1} - s_j)^{-1 + \epsilon}\prod_{\substack{1 \leq i \leq k \\ i \neq j}} (s_{i + 1} - s_i)^{-\frac{1}{2}} . 
  \end{align*}
 \end{lemma}

Now, we are ready to present the proof of Proposition \ref{thm_rate-l2}. 
\begin{proof}[Proof of Proposition \ref{thm_rate-l2}]
This proposition is proved using the chaos expansion. In view of the hypercontractivity property for fixed chaoses (cf. \cite[Formula (7.1.41)]{ws-17-hu}), it is sufficient to establish the result for $p = 2$. The extension to $p > 2$ are left to the reader for brevity. 

Thanks to Theorem \ref{thm_eu}, Lemmas \ref{lmm-chaos-h} and \ref{prop_l2-f-g}, we can write
\begin{align*}
 \bE \big[ \big|u_h ( \lfloor t \rfloor_h ,  \lfloor x \rfloor_{2 \sqrt{h}}) & - u(t, x) \big|^2 \big] = \sum_{k = 1}^{\infty} \frac{1}{k!}\int_{[0, t]^k} \int_{\bR^k} \big| f_k (\bfs_k, \bfy_k) - g_{h, k} (\bfs_k, \bfy_k)\big|^2 d \bfs_k d \bfy_k\\
 \leq & h^{\frac{1}{2} - \epsilon} \sum_{k = 1}^{\infty} C^k \int_{[0, t]_<^k} \sum_{j = 1}^k (s_{j + 1} - s_j)^{-1 + \epsilon}\prod_{\substack{1 \leq i \leq k \\ i \neq j}} (s_{i + 1} - s_i)^{-\frac{1}{2}}  d \bfs_k. 
\end{align*}
Using \cite[Lemma 4.5]{ejp-15-hu-huang-nualart-tindel} and upper bounds for the Mittag-Leffler functions (cf. \cite[Formula (1.8.10)]{elsevier-06-kilbas-strivastava-trujillo}), we conclude that
\begin{align*}
 \bE \big[ \big|u_h ( \lfloor t \rfloor_h ,  \lfloor x \rfloor_{2 \sqrt{h}}) - u(t, x) \big|^2 \big]\lesssim & h^{\frac{1}{2} - \epsilon} \sum_{k = 1}^{\infty} \frac{C^k t^{k/2 - \frac{1}{2} + \epsilon}}{\Gamma(k/2 + 1/2 + \epsilon)} \lesssim h^{\frac{1}{2} - \epsilon} e^{C t}. 
\end{align*}
This completes the proof of this proposition. 
\end{proof}

It remains to show Lemmas \ref{lmm-chaos-h}--\ref{prop_l2-f-g}. 

\begin{proof}[Proof of Lemma \ref{lmm-chaos-h}]
 Let $\{H_{k} \colon k \in \bZ_{>0} \}$ denote the Hermite polynomials as defined in \eqref{def_hermite}. Then,  using \eqref{def_wick-herm}, $u_h (m, n)$ given by \eqref{fk_h_X} can be written as 
 \begin{align*}
  u_h (m, n) = & 1 + \sum_{k = 1}^{\infty} \frac{t^k}{k!} \bE^S \big[H_k (X/t)\big], 
 \end{align*}
where
\begin{align*}
 X \coloneqq \sum_{i = 1}^{m} W_h \Big(i, \Big\lfloor\frac{S_{m + 1 - i} }{2} \Big\rfloor + n \Big) \quad \text{and} \quad t = \sqrt{{\rm Var}^W (X)}. 
\end{align*}
Additionally, \cite[Example 5.6]{ws-17-hu} suggests that, for any $k \in \bZ_{>0}$, and $\phi \in L_2 (\bR_{\geq 0} \times \bR)$, 
 \begin{align*}
   H_{k} \big(W(\phi)\big) = I_{k} \big(\phi^{\otimes k}\big). 
 \end{align*}
 Recalling $W_h$ given by \eqref{def_wh} and that $\cT_h^{(m)}$ and $\cS_h^{(n)}$ are indicator functions on intervals
$[(m - 1)h, mh )$ and $[2n \sqrt{h}, 2 (n + 1) \sqrt{h} )$, receptively, we can write
 \begin{align*}
  H_{k} (X/t) 
 = & t^{-k} I_k \bigg( \Big(\sum_{i = 1}^m \frac{1}{2\sqrt{h}} \cT_h^{(i)} \times \cS_h^{\big(\lfloor\frac{S_{m + 1 - i}}{2}\rfloor + n \big)} \Big)^{\otimes k} \bigg). 
 \end{align*}
 Therefore, from the preceding definitions of the integral kernels  $g_{h, k}$'s in \eqref{def_gk}, and of the alternating indicator $\tau$ in \eqref{def_tau-t}, it follows that for all $(\bfs_k, \bfy_k) \in [0, t]^k \times \bR^k$
 \begin{align*} 
  g_{h, k} (\bfs_k, \bfy_k) = &  \bE^S \bigg[\prod_{j = 1}^k \Big(\sum_{i = 1}^m \frac{1}{2\sqrt{h}} \cT_h^{(i)} (s_j) \times \cS_h^{\big(\lfloor\frac{S_{m + 1 - i} }{2}\rfloor + n \big)} (y_j) \Big)\bigg] \nonumber\\
   = & 2^{-k} h^{- \frac{k}{2}} \bP \Big(\Big\lfloor\frac{S_{m - \lfloor s_j \rfloor_h} }{2}\Big\rfloor = \lfloor y_j \rfloor_{2 \sqrt{h}} - n, j = 1, \dots, k \Big) \nonumber\\
  = & 2^{-k} h^{- \frac{k}{2}} \bP \big(S_{m - \lfloor s_j \rfloor_h} = 2 \big(\lfloor y_j \rfloor_{2 \sqrt{h}} - n \big) + \tau \big(m - \lfloor s_j \rfloor_h \big), j = 1, \dots, k \big). 
 \end{align*}
Thus, it follows from the property that $S$ is a homogeneous Markov chain with symmetric transition probabilities that
\[
  g_{h, k} (\bfs_k, \bfy_k) = 2^{-k} h^{- \frac{k}{2}} \prod_{i = 1}^{k} \bP \big(S_{\ks_i} = \ky_i\big) = 2^{-k} h^{- \frac{k}{2}} \prod_{i = 1}^{k} G \big(\ks_i, \ky_i \big),
\]
where we recall that $\ks_i$ and $\ky_i$ are defined in \eqref{def_ks} and \eqref{def_ky}, respectively.
 This completes the proof of the lemma. 
\end{proof}

\begin{remark}\label{rmk-chaos-colour}
 One may observe that the space-time white noise assumption in the proof of Lemma \ref{lmm-chaos-h} is not essential. In fact, the result remains valid in the case of coloured noise, as will be applied in the next subsection. 
\end{remark}

\begin{proof}[Proof of Lemma \ref{lmm-l2-g-p}]

Denote by $L$ the left hand side of \eqref{eq_l2-g-p}. First, when $t < h$, then it simply follows from the triangular inequality, and Lemmas \ref{lmm_g-h1} and \ref{lmm_inn_heat}, that
\begin{align*}
 L \lesssim & \frac{1}{4h} \int_{\bR} G \big(\lfloor t \rfloor_{h}, 2 (\lfloor x \rfloor_{2 \sqrt{h}} - \lfloor y \rfloor_{2 \sqrt{h}})\big)^2 dy + \int_{\bR} p_{t} (x - y)^2 d y \\
 \lesssim & t^{-\frac{1}{2}} \leq t^{-\frac{1}{2}} \big(h/t\big)^{\frac{1}{2} - \epsilon} = h^{\frac{1}{2} - \epsilon} t^{-1 + \epsilon}, 
\end{align*}
for all $\epsilon < \frac{1}{2}$, since $t < h$. 
This completes the proof of inequality \eqref{eq_l2-g-p} assuming $t < h$.

In the remainder of this proof, we focus on the case where $t \geq h$.  Thanks to Lemma \ref{tail-normal} and Corollary \ref{coro_g-p}, we deduce that
\begin{align*}
 L \lesssim h t^{-2} \int_{x - \frac{3t}{\sqrt{h}}}^{x + \frac{3t}{\sqrt{h}}} dy + t^{-\frac{1}{2}}\int_{\bR \setminus [x - \frac{t}{\sqrt{h}}, x + \frac{t}{\sqrt{h}}]} p_{2t} (x - y) dy 
 \lesssim t^{-1}h^{\frac{1}{2}} \leq t^{-1 + \epsilon} h^{\frac{1}{2} - \epsilon}, 
\end{align*}
because $h \leq t$. 
This concludes the verification of the lemma. 
\end{proof}

\begin{proof}[Proof of Lemma \ref{prop_l2-f-g}]
 The proof of this lemma is based on Lemma \ref{lmm-l2-g-p} and the induction in $k$. When $k = 1$, this lemma reduces to Lemma \ref{lmm-l2-g-p}. For general $k \geq 2$, we can write
 \begin{align*}
  \int_{\bR^k} d \bfy_k \big| f_k (\bfs_k, \bfy_k) - g_{h, k} (\bfs_k, \bfy_k)\big|^2 \lesssim & J_1 + J_2, 
 \end{align*}
 where
 \begin{align*}
  J_1 \coloneqq \int_{\bR} d y_k 
\frac{1}{4h} G (\ks_k, \ky_k)^2 \int_{\bR^{k - 1}} d \bfy_{k - 1} \Big[ & (4h)^{-\frac{k - 1}{2}}\prod_{j = 1}^{k - 1} G(\ks_j, \ky_j) - \prod_{j = 1}^{k - 1} p_{s_{j + 1} - s_{j}} (y_{ j + 1} - y_{j}) \Big]^2, 
 \end{align*} 
 and
 \begin{align*}
  J_2 \coloneqq & \int_{\bR} d y_k \Big( \frac{1}{2\sqrt{h}} G (\ks_k, \ky_k) - p_{t - s_{k}} (x - y_{ k}) \Big)^2 \int_{\bR^{k - 1}} d \bfy_{k - 1} \prod_{j = 1}^{k - 1} p_{s_{ j + 1} - s_{j}} (y_{ j + 1} - y_{j})^2 . 
 \end{align*} 
 The proof of this lemma is complete by referring to Lemmas \ref{lmm_g-h1} and \ref{lmm-l2-g-p}, inequality \eqref{eq_est-chaos-flat}, and the induction hypothesis. 
\end{proof}

\subsection{Coloured noises}\label{sec_colour}
In this subsection, we complete the proof of Theorem \ref{thm_rate-frac}. 
The proof closely follows that of the white noise case, relying heavily on the chaos expansion for approximate solutions (Lemma \ref{lmm-chaos-h}) and on the difference estimate between the densities of Brownian motion and rescaled random walk, as summarised in the next lemma. 

\begin{lemma}\label{prop_frac-f-g}
 Fix $h > 0$, $k \in \bZ_{>0}$, and $(s_1, \dots, s_k) \in [0, t]_<^k$. Let $f_k$ and $g_{h, k}$ be given as in \eqref{def_fk} and \eqref{def_gk}, respectively. Then, with any $\epsilon \in (0, (2H + H_* - 1)\wedge 1)$, there exists a constant $C > 0$ depending on $\epsilon$ and $t$ such that
 \begin{align*}
  \big\| f_k (\bfs_k, \cdot) - g_{h, k} (\bfs_k, \cdot) \big\|_{\cH_*^{\otimes k}}^2 
   \leq C^k h^{(2H + H_* - 1) \wedge 1 - \epsilon} \sum_{j = 1}^k \Big[ &(s_{j + 1} - s_j)^{-[2H \wedge (2 - H_*)] + \epsilon}\\
   & \times \prod_{\substack{1 \leq i \leq k \\ i \neq j}} (s_{i + 1} - s_i)^{H_* - 1} \Big]. 
 \end{align*}
\end{lemma}

\begin{proof} 
 Once establishing Lemma \ref{lmm-frac-g-p} below, Lemma \ref{prop_frac-f-g} follows easily by the same argument used in the proof of Lemma \ref{prop_l2-f-g}. Thus, we omit the details for conciseness. 
\end{proof}

\begin{lemma}\label{lmm-frac-g-p}
 Fix $h > 0$. Let $p$ denote the heat kernel, and let $G$ be defined in \eqref{def_Gh}. Then, for any $(t, x) \in \bR_{> 0} \times \bR$, and $\epsilon \in (0, (2H + H_* - 1) \wedge 1)$, it holds with the alternating indicator $\tau$ given by \eqref{def_tau-t} that
  \begin{align*}
   &\Big\|\frac{1}{2\sqrt{h}} G \big(\lfloor t \rfloor_h, 2(\lfloor x \rfloor_{2 \sqrt{h}} - \lfloor \cdot \rfloor_{2 \sqrt{h}}) \pm \tau(\lfloor t \rfloor_h) \big) - p_{t} (x - \cdot)\Big\|_{\cH_*}^2 \\
    \lesssim & h^{(2 H + H_* - 1) \wedge 1 - \epsilon} t^{- [2H \wedge (2 - H_*)]+ \epsilon} , 
  \end{align*}
where `$\pm$' indicates the inequality holds with either `$+$' or `$-$'. 
 \end{lemma}

The proof of Lemma \ref{lmm-frac-g-p} is postponed to the end of this subsection. 

\begin{proof}[Completing the proof of Theorem \ref{thm_rate-frac}]
Similar to the white noise case, we only consider the case when $p = 2$, and the proof is a straightforward consequence of Lemma \ref{prop_frac-f-g}. Here, we assume that $H > \frac{1}{2}$, as when $H = \frac{1}{2}$, the final inequality in \eqref{ieq_rate-frac-1} below can be obtained directly without using the Hardy-Littlewood-Sobolev inequality (Lemma \ref{lmm_hls}), and the desired result follows accordingly. 

It follows from Theorem \ref{thm_eu}, Lemmas \ref{lmm_hls}, \ref{lmm-chaos-h}, and \ref{prop_frac-f-g}, and Remark \ref{rmk-chaos-colour}, that
\begin{align}\label{ieq_rate-frac-1}
 & \bE \big[ |u_h ( \lfloor t \rfloor_h ,  \lfloor x \rfloor_{2 \sqrt{h}}) - u(t, x) |^2 \big] \nonumber\\
 = & \sum_{k = 1}^{\infty} \frac{C_H^k}{k!} \int_{[0, t]^{2k}} \int_{\bR^k} \prod_{i = 1}^k |s_i - r_i|^{2H - 2} \big\|f_k (\bfr_k, \cdot) - g_{h, k} (\bfr_k, \cdot)\big\|_{\cH_*^{\otimes k}} d \bfs_kd \bfr_k \nonumber\\
 \leq & h^{(2 H + H_* - 1) \wedge 1 - \epsilon} \sum_{k = 1}^{\infty} C^k\bigg[ \int_{[0, t]_<^k} \Big(\sum_{j = 1}^k (s_{j + 1} - s_j)^{- 1\wedge \frac{2 - H_*}{2H} + \frac{\epsilon}{2H}}\prod_{\substack{1 \leq i \leq k \\ i \neq j}} (s_{i + 1} - s_i)^{\frac{H_* - 1}{2H}} \Big)d \bfs_k \bigg]^{H}. 
\end{align}
The proof of this theorem is then complete by invoking \cite[Lemma 4.5]{ejp-15-hu-huang-nualart-tindel} and upper bounds for the Mittag-Leffler functions. 
\end{proof}

We will conclude this subsection by providing the proof of Lemma \ref{lmm-frac-g-p}. 

\begin{proof}[Proof of Lemma \ref{lmm-frac-g-p}]
It suffices to prove the lemma for the case when $\pm$ takes the sign $+$, as the other case follows similarly. 
Denote
\begin{align*}
 L \coloneqq \Big\|\frac{1}{2\sqrt{h}} G \big(\lfloor t \rfloor_h, 2(\lfloor x \rfloor_{2 \sqrt{h}} - \lfloor \cdot \rfloor_{2 \sqrt{h}}) + \tau(\lfloor t \rfloor_h)\big) - p_{t} (x - \cdot)\Big\|_{\cH_*}^2 . 
\end{align*}
If $t < h$,  thanks to Lemmas \ref{lmm_g-h1} and \ref{lmm_inn_heat}, we have
\begin{align*}
 L \lesssim & \Big\|\frac{1}{2\sqrt{h}} G \big(\lfloor t \rfloor_h, 2(\lfloor x \rfloor_{2 \sqrt{h}} - \lfloor \cdot \rfloor_{2 \sqrt{h}})\big)\Big\|_{\cH_*}^2 + \|p_{t} (x - \cdot)\|_{\cH_*} \\
\lesssim & t^{H_* - 1} \lesssim h^{(2 H + H_* - 1) \wedge 1 - \epsilon} t^{- [2H \wedge (2 - H_*)] + \epsilon}, 
\end{align*}
for all $\epsilon \in (0, (2H + H_* - 1)\wedge 1)$.

Next, assume $t \geq h$.  Similar to the proof of Lemma \ref{lmm-l2-g-p}, we only need to verify this lemma in case that $t/h$ is a positive even integer. Let
\[ 
D_1 \coloneqq \big\{y \colon 2\big|\lfloor x \rfloor_{2 \sqrt{h}} - \lfloor y \rfloor_{2 \sqrt{h}}\big| \leq (t/h)^{\alpha} \big\} \quad \text{and} \quad
 D_2 \coloneqq \bR \setminus D_1, 
 \]
with $\alpha \coloneqq (\frac{1 - H}{H_*} \vee \frac{1}{2} + \frac{\epsilon}{2H_*}) \wedge 1$. 
Then, it can be easily verified that
\[
 D_1 \subseteq D_1' \coloneqq \big\{y \colon |x - y| \leq 4 h^{\frac{1}{2} - \alpha} t^{\alpha} \big\}  \quad \text{and} \quad  D_2 \subseteq D_2' \coloneqq \big\{y \colon |x - y| > h^{\frac{1}{2} - \alpha} t^{\alpha}\big\}. 
\]
As a result, we can write
\begin{align*}
L \lesssim L_1 + L_2, 
\end{align*}
where
\[
 L_i \coloneqq \bigg\| \Big(\frac{1}{2\sqrt{h}} G \big(\lfloor t\rfloor_h, 2 (\lfloor x \rfloor_{2 \sqrt{h}} - \lfloor \cdot \rfloor_{2 \sqrt{h}})\big) - p_{t} (x - \cdot)\Big) \1_{D_i'} (\cdot)\bigg\|_{\cH_*}^2, \quad i = 1, 2.
\]
Recall that $\alpha = (\frac{1 - H}{H_*} \vee \frac{1}{2} + \frac{\epsilon}{2H_*}) \wedge 1$ and $t \geq h$.  As a result of Corollary \ref{coro_g-p}, it holds that
\begin{align} \label{eq_i-1-frac}
 L_1 \lesssim & h t^{-2} \int_{D_1' \times D_1'} |y - z|^{2H_* - 2} \lesssim h^{H_* - 2\alpha H_* + 1} t^{2 \alpha H_* - 2}\nonumber \\
  \lesssim & h^{(2 H + H_* - 1) \wedge 1 - \epsilon} t^{- [2H \wedge (2 - H_*)] + \epsilon}. 
\end{align}

It suffices to estimate $L_2$. Using the elementary triangular inequality, we can write
\begin{align*}
 L_2 \lesssim L_{2, 1} + L_{2, 2}, 
\end{align*}
with
\begin{align*}
 L_{2, 1} \coloneqq \Big\|\frac{1}{2\sqrt{h}} G \big(\lfloor t\rfloor_h, 2 (\lfloor x \rfloor_{2 \sqrt{h}} - \lfloor \cdot \rfloor_{2 \sqrt{h}})\big)\1_{D_2'} (\cdot)\Big\|_{\cH_*}^2 \quad \text{and} \quad L_{2, 2} \coloneqq \| p_{t} (x - \cdot) \1_{D_2'} (\cdot)\|_{\cH_*}^2.
\end{align*}
A changing of variables $(y, y - z) \mapsto (y, z')$ yields that
\begin{align*}
 L_{2, 2} = & \int_{D_2' \times D_2'} dy dz p_t(x - y) p_t (x - z) |y - z|^{2H_* - 2}\\ 
 \lesssim & \int_{D_2'} dy p_t(x - y) \int_{\bR} d z' |z'|^{2H_* - 2} p_t(x - y + z'). 
\end{align*}
Notice that for any $y \in D_2'$, 
\begin{align*}
   \int_{\bR} d z' |z'|^{2H_* - 2} p_t(x - y  z') \lesssim \int_{\bR} |\xi|^{1 - 2H_*} e^{- \frac{1}{2} t |\xi|^2} \lesssim t^{H_* - 1} \lesssim h^{H_* - 1}, 
\end{align*}
because $t \geq h$ and $H_* < 1$. Thus it follows from Lemma \ref{tail-normal}, and the inequality that $\sup_{x \geq 1} e^{-\frac{x}{2}} x^{\beta} < \infty$ for all $\beta \in \bR$, 
\begin{align*}
 L_{2, 2} \lesssim & h^{H_* - 1} \times (t/h)^{\frac{1}{2} - \alpha} e^{-\frac{1}{2}(t/h)^{2\alpha - 1}} \lesssim h^{H_* - 1} (t/h)^{\frac{1}{2} - \alpha - \beta (2\alpha - 1)}. 
 \end{align*}
Choose
\[
 \beta = \frac{ [2H \wedge (2 - H_*)] - \epsilon}{2\alpha - 1} - \frac{1}{2}, 
\]
which is finite, because $\alpha = (\frac{1 - H}{H_*} \vee \frac{1}{2} + \frac{\epsilon}{2H_*}) \wedge 1 > \frac{1}{2}$. 
Then, we get
\begin{align}\label{eq_i-2-2-frac}
 L_{2, 2} \lesssim h^{(2 H + H_* - 1) \wedge 1 - \epsilon} t^{- [2H \wedge (2 - H_*)] + \epsilon}. 
\end{align}
The estimate for $L_{2, 1}$ is also similar. Write
\begin{align*}
 L_{2, 1} = & \int_{D_2' \times D_2'} dy dz |y - z|^{2H_* - 2} \frac{1}{2\sqrt{h}} G \big(\lfloor t\rfloor_h, 2 (\lfloor x \rfloor_{2 \sqrt{h}} - \lfloor y \rfloor_{2 \sqrt{h}})\big) \\ 
 & \qquad \times \frac{1}{2\sqrt{h}} G \big(\lfloor t\rfloor_h, 2 (\lfloor x \rfloor_{2 \sqrt{h}} - \lfloor z \rfloor_{2 \sqrt{h}})\big) \\
 \lesssim & \int_{D_2'} dy \frac{1}{2\sqrt{h}} G \big(\lfloor t \rfloor_h, 2 (\lfloor x \rfloor_{2 \sqrt{h}} - \lfloor y \rfloor_{2 \sqrt{h}})\big) \\ 
 & \qquad \times \int_{\bR} dz' |z'|^{2H_* - 2}  \frac{1}{2\sqrt{h}} G \big(\lfloor t \rfloor_h, 2 (\lfloor x \rfloor_{2 \sqrt{h}} - \lfloor y - z'\rfloor_{2 \sqrt{h}})\big). 
\end{align*}
Next, it follows from Lemma \ref{tail-binom} and Corollary \ref{coro_g-p} that
\begin{align*}
 L_{2, 1} \lesssim & \bP \big(S_{t/h} > (t/h)^{\alpha}\big) \times \Big(\sup_{y \in \bR} \int_{\bR} dz' |z'|^{2H_* - 2} p_t(y - z') + \frac{\sqrt{h}}{t} \int_{- \frac{3 t}{\sqrt{h}}}^{\frac{3 t}{\sqrt{h}}} dz' |z'|^{2H_* - 2}\Big)\\
 \lesssim & e^{- \frac{1}{3} (t/h)^{2\alpha - 1}} \big(t^{H_* - 1} + t^{2H_* - 2} h^{1 - H_*}\big) \lesssim (t/h)^{- \beta(2\alpha - 1)} h^{H_* - 1}, 
\end{align*}
for all $\beta \in \bR$. Set 
\[
 \beta = \frac{ [2H \wedge (2 - H_*)] - \epsilon}{2\alpha - 1}. 
\]
Then, we end up with
\begin{align}\label{eq_i-2-1-frac}
 L_{2, 1} \lesssim h^{(2 H + H_* - 1) \wedge 1 - \epsilon} t^{- [2H \wedge (2 - H_*)] + \epsilon}. 
\end{align}
The proof of this lemma is complete by combining inequalities \eqref{eq_i-1-frac}--\eqref{eq_i-2-1-frac}. 
\end{proof}

\section{\eqref{pam} with delta initial condition}\label{sec_delta}
Let $z \in \bR$. Recall that by definition, a random field $u_z = \{u_z (t, x) \colon (t, x) \in \bR_{\geq 0} \times \bR\}$ is called a solution to \eqref{pam} with a delta initial condition at $z \in \bR$, if
\begin{align*}
 u_z(t, x) = p_t (x - z) + \int_0^t \int_{\bR} p_{t - s} (x - y) u_z (s, y) W(ds, dy). 
\end{align*}
Therefore, for any signed Radon measure $\mu$ on $\bR$, if 
\[
 u_{\mu} (t, x) \coloneqq \int_{\bR} u_z(t, x) \mu(dz), 
\]
is well-defined in certain sense, $u_{\mu}$ is a solution to \eqref{pam} with the initial condition $\mu$. This can be seen with an (unrigorous) application of the stochastic Fubini theorem, 
\begin{align*}
 u_{\mu} (t, x) = & \int_{\bR} p_t (x - z) \mu(dz) + \int_{\bR} \int_0^t \int_{\bR} p_{t - s} (x - y) u_z (s, y) W(ds, dy) \mu(dz)\\
 = & \int_{\bR} p_t (x - z) \mu(dz) + \int_0^t \int_{\bR} p_{t - s} (x - y) u_{\mu} (s, y) W(ds, dy). 
\end{align*}
In other words, solving \eqref{pam} with delta initial conditions is crucial for understanding the solution to \eqref{pam} with general initial conditions. Therefore, we focus on approximating the solution to \eqref{pam} with delta initial condition at $z\in \bR$. For simplicity, we further assume that $z = 0$, as other cases are essentially the same. 

With a delta initial condition at the origin, the Wick renormalised Feynman-Kac formula for \eqref{pam} can be written as
\begin{align}\label{eq_fk-delta}
 u(t, x) = u_0(t, x) = p_t(x) \bE^{\cB} \bigg[ \wick \Big(\int_0^t \int_{\bR} \delta (\cB_{t - s}^{(t, -x)} + x - y) W(ds, dy)\Big) \bigg], 
\end{align}
where $\wick$ is given by \eqref{def_wick}, and $\cB^{(t, x)}$ denotes a Brownian bridge with $\cB^{(t, x)}_0 = 0 $ and $\cB_t^{(t, x)} = x$ for any $(t, x) \in \bR_{> 0} \times \bR$. Hence, it is natural to formulate the discrete Feynman-Kac approximation using a random walk bridge, as it converges in distribution to a Brownian bridge (cf. \cite[Corollary on page 568]{jmm-68-liggett}).

\begin{remark}
To the best of our knowledge, formula \eqref{eq_fk-delta} has not appeared explicitly in the literature. Similar to the flat initial condition case, it may introduce an `$\infty$' in the exponent, and therefore must be understood via a suitable approximation.  Since we only use it here to motivate our approximating solution \eqref{fk_h_X-delta} below which is well-defined, we do not pursue a full rigorous treatment. For further intuition on \eqref{eq_fk-delta}, we refer the reader to the Feynman-Kac formula for moments in terms of Brownian bridges in \cite[Section 4]{aihp-17-huang-le-nualart}.
\end{remark}

For any $(m, n) \in \bZ_{> 0} \times \bZ$, such that $|n| \leq m$ and $n + m$ is an even number, let $\bS^{(m, n)} = \{ \bS_{i}^{(m, n)} \colon i = 0, 1, \dots, m \}$ be a random walk bridge with $\bS^{{(m, n)}}_0 = 0$, and $\bS_{m}^{(m, n)} = n$. 
 Then, the joint distribution of $\bS^{(m, n)}$ can be expressed as follows, 
\begin{align}\label{joint-rwb}
 \bP \big(\bS_{j_1}^{(m, n)} = i_1, \dots, \bS^{(m, n)}_{j_k} = i_k \big) = &  \bP \big(S_{j_1} = i_1, \dots, S_{j_k} = i_k \big| S_m = n \big), 
\end{align}
for all $0 \leq j_1 < \dots < j_k \leq m$, where $S$ is a simple random walk. 
Next, we define the discrete Feynman-Kac approximate solution $u_h (m, n)$ given by
\begin{align}\label{fk_h_X-delta}
 u_h (m, n) = & \frac{1}{2 \sqrt{h}} G \big(m, n - \tau (m)\big) \bE^{\bS} \Bigg[\wick \Bigg(\sum_{i = 1}^{m} W_h \bigg(i, \Big\lfloor\frac{\bS_{m + 1 - i}^{(m, \tau(m) - 2n)}}{2} \Big\rfloor + n \bigg) \Bigg) \Bigg], 
\end{align}
where $\tau$ is defined in \eqref{def_tau-t}.

\begin{theorem}\label{thm_rate-delta}
 Let $u$ be the solution to \eqref{pam} with a delta initial condition at the origin, driven by a fractional Brownian sheet with Hurst parameters $H \geq \frac{1}{2}$ and $H_* \geq \frac{1}{2}$ in time and space, respectively. Let $u_h = \{u_h(m, n) \coloneqq (m, n) \in \bZ_{> 0}\times \bZ\}$ be defined in \eqref{fk_h_X-delta} with some $h > 0$. Then, for any $p \geq 2$, $(t, x) \in \bR_{> 0} \times \bR$ and $\epsilon \in (0, (2H + H_* - 1) \wedge 1 )$, 
 \begin{align}\label{ieq_rate-frac-delta}
   \big\| u_h \big( \lfloor t/h \rfloor ,  \big\lfloor x/\sqrt{4h} \big\rfloor \big) - u(t, x) \big\|_{L^p(\Omega)}^2 \lesssim h^{(2 H + H_* - 1)\wedge 1 - \epsilon}, 
 \end{align}
 where the implicit constant depends on $p$, $t$ and $\epsilon$. 
\end{theorem}

\begin{remark}
Though inequalities \eqref{ieq_rate-frac} and \eqref{ieq_rate-frac-delta} share the same form, they differ in their implicit constants. Specifically, the constant in \eqref{ieq_rate-frac} can be chosen uniformly for $t$ in any compact interval $[0,T]$. By contrast, because the solution to \eqref{pam} with a delta initial condition blows up as $t \downarrow 0$, one must restrict $t$ away from $0$ to make \eqref{ieq_rate-frac-delta} hold, and accordingly, the implicit constant in \eqref{ieq_rate-frac-delta} diverges to $\infty$ as $t \downarrow 0$.
\end{remark}

\subsection{Proof of Theorem \ref{thm_rate-delta}}
In this subsection, we provide the proof of Theorem \ref{thm_rate-delta}, which follows from the following lemmas. The proofs can be found in Subsection \ref{ss_prf-lmm-delta}.

\begin{lemma}\label{lmm-chaos-h-delta}
 Let $u_h = \{u_h (m, n) \colon (m, n) \in \bZ_{> 0} \times \bZ\}$ be given by \eqref{fk_h_X-delta}. Then, $u_h$ can be represented as the following chaos expansion, 
 \begin{align}\label{chaos-uh-delta}
 u_h (m, n) = & \frac{1}{2 \sqrt{h}} G \big(m, 2n - \tau(m)\big) + \sum_{k = 1}^{\infty} \frac{1}{k!} I_k(g_{h, k}), 
 \end{align}
 where $\tau$ is defined in \eqref{def_tau-t}, and $g_{h, k} \colon [0, mh]^k \times \bR^k \to \bR$ is given by
 \begin{align}\label{def_gk-delta}
  g_{h, k} (\bfs_k, \bfy_k) =  g_{h, k}^{t,x} (\bfs_k, \bfy_k) = & 2^{-(k + 1)} h^{- \frac{k + 1}{2}} \prod_{j = 0}^k G\big(\ks_j, \ky_j \big), 
 \end{align}
 with $G$, $(\ks_0, \dots, \ks_k)$ and $(\ky_0, \dots, \ky_k)$ defined by \eqref{def_Gh}, \eqref{def_ks} and \eqref{def_ky}, respectively, with the parameters $(t, x) \in \bR_{> 0} \times \bR$, such that $\lfloor t \rfloor_h = m$ and $\lfloor x \rfloor_{2 \sqrt{h}} = n$. 
\end{lemma}

\begin{lemma}\label{lmm-frac-g-p-delta}
 Fix a time horizon $[0, T]$, and fix $h > 0$. Let $p$ denote the heat kernel, and let $G$ be defined in \eqref{def_Gh}. Then, for any $0 < s < t \leq T$, and $x \in \bR$, the following inequality hold with any $\epsilon \in (0, (2H + H_* - 1) \wedge 1)$ and an implicit constant depending on $\epsilon$ and $T$, 
  \begin{align*} 
   &\Big\| \frac{1}{4h} G \big(\lfloor t \rfloor_h - \lfloor s \rfloor_h, 2 (\lfloor x \rfloor_{2 \sqrt{h}} - \lfloor \cdot \rfloor_{2 \sqrt{h}}) \pm \tau(\lfloor t \rfloor_h - \lfloor s \rfloor_h)\big) G \big(\lfloor s \rfloor_h, 2\lfloor \cdot \rfloor_{2 \sqrt{h}} \pm \tau(\lfloor s \rfloor_h)\big) \\
   &\quad - p_{t - s} (x -\cdot) p_{s} (\cdot)\Big\|_{\cH_*}^2
    \lesssim  \big(s (t - s) \big)^{- 2H \wedge (2 - H_*)+ \epsilon/2} h^{(2 H + H_* - 1) \wedge 1 - \epsilon}, 
  \end{align*}
where $\tau(t)$ is given by \eqref{def_tau-t}, and `$\pm$' indicates the inequality holds with either `$+$' or `$-$'. 
 \end{lemma}

\begin{lemma}\label{prop_f-g-delta}
 Let $k \in \bZ_{>0}$, and let $(s_1, \dots, s_k) \in [0, t]_<^k$. Let $f_k = f_k^{t, x}$ and $g_{h, k} = g_{h, k}^{t, x}$ be given as in \eqref{def_fk-delta} and \eqref{def_gk-delta} with some $h > 0$, respectively. Then, with any $\epsilon \in (0, 2H + H_* - 1)$, there exists a constant $C > 0$, such that
\begin{align}\label{f-g-frac-delta}
  \big\| f_k (\bfs_k, \cdot) & - g_{h, k} (\bfs_k, \cdot) \big\|_{\cH_*^{\otimes k}}^2 
   \leq C^k h^{(2H + H_* - 1) \wedge 1 - \epsilon} \nonumber\\
   & \times \bigg(\sum_{j = 1}^k \Big[ (s_{j + 1} - s_j)^{-[2H \wedge (2 - H_*)]+ \epsilon} s_j^{-H_*}\prod_{\substack{0 \leq i \leq k \\ i \neq j}} (s_{i + 1} - s_i)^{H_* - 1} \Big] \nonumber\\
& \qquad + \prod_{i = 2}^k (s_{i + 1} - s_i)^{H_* - 1} (s_2 - s_1)^{-[2H \wedge (2 - H_*)] + \epsilon/2} s_1^{-[2H \wedge (2 - H_*)] + \epsilon/2} \bigg). 
 \end{align}
\end{lemma}

Now, we are ready to present the proof of Theorem \ref{thm_rate-delta}. 

\begin{proof}[Proof of Theorem \ref{thm_rate-delta}]
As in the case of the flat initial condition, we prove the theorem for  $p = 2$, and leave the extension to general $p > 2$ to the reader.
Due to Theorem \ref{thm_eu}, Lemmas \ref{lmm_hls}, \ref{lmm-chaos-h-delta}, and \ref{prop_f-g-delta}, 
\begin{align*}
 & \bE \big[ |u_h ( \lfloor t \rfloor_h ,  \lfloor x \rfloor_{2 \sqrt{h}}) - u(t, x) |^2 \big] 
 \leq h^{(2 H + H_* - 1) \wedge 1 - \epsilon} \sum_{k = 1}^{\infty} C^k \big(J_{k, 1} (t) + J_{k, 2}(t)\big)^{H}, 
\end{align*}
where by applying \cite[Lemma 4.5]{ejp-15-hu-huang-nualart-tindel}, 
\begin{align*}
 J_{k, 1}(t) \coloneqq & \int_{[0, t]_<^k}\sum_{j = 1}^k (s_{j + 1} - s_j)^{- (1 \wedge \frac{2 - H_*}{2 H})+ \frac{\epsilon}{2H}} s_j^{- \frac{H_*}{2H}} \prod_{\substack{0 \leq i \leq k \\ i \neq j}} (s_{i + 1} - s_i)^{\frac{H_* - 1}{2H}} d \bfs_k\\
\leq & \frac{C^k t^{- \frac{H_* \wedge (2 - 2H)}{2 H} + \frac{(k - 1) (2 H + H_* - 1)}{2H} + \frac{\epsilon}{2H}}}{\Gamma \big(1 - \frac{H_* \wedge (2 - 2H)}{2 H} + \frac{(k - 1) (2 H + H_* - 1)}{2H} + \frac{\epsilon}{2H} \big)}, 
\end{align*}
and
\begin{align*}
 J_{k, 2}(t) \coloneqq & \int_{[0, t]_<^k} \prod_{i = 2}^k (s_{i + 1} - s_i)^{\frac{H_* - 1}{2 H}} (s_2 - s_1)^{- (1\wedge \frac{2 - H_*}{2H}) + \frac{\epsilon}{4H}} s_1^{- (1 \wedge \frac{2 - H_*}{2 H}) + \frac{\epsilon}{4 H}} d \bfs_k\\
\leq & \frac{C^k t^{- \frac{H \wedge (2 - H - H_*)}{ H} + \frac{(k - 1) (2 H + H_* - 1)}{2H} + \frac{\epsilon}{2 H}}}{\Gamma \big(1 - \frac{H \wedge (2 - H - H_*)}{ H} + \frac{(k - 1) (2 H + H_* - 1)}{2H} + \frac{\epsilon}{2 H}\big)} . 
\end{align*}
Then, this theorem is a result of upper bounds for the Mittag-Leffler functions. 
\end{proof}

\subsection{Proof of Lemmas \ref{lmm-chaos-h-delta}--\ref{prop_f-g-delta}}\label{ss_prf-lmm-delta}

\begin{proof}[Proof of Lemma \ref{lmm-chaos-h-delta}]
Analogously to the proof of Lemma \ref{lmm-chaos-h}, and taking account of \eqref{joint-rwb} and \eqref{fk_h_X-delta}, we can show that the integral kernels in \eqref{chaos-uh-delta} can be represented as
 \begin{align*} 
 g_{h, k} (\bfs_k, \bfy_k)  
 = &  \frac{1}{2 \sqrt{h}} G \big(m, n - \tau (m)\big) 2^{-k} h^{- \frac{k}{2}} \bP \bigg(\Big\lfloor\frac{S_{m - \lfloor s_j \rfloor_h} }{2}\Big\rfloor = \lfloor y_j \rfloor_{2 \sqrt{h}} - n, j = 1, \dots, k \bigg| \Big\lfloor\frac{S_{m} }{2}\Big\rfloor = - n\bigg) \\
  = & 2^{-(k + 1)}h^{- \frac{k + 1}{2}} 
   \bP \Big( S_{m - \lfloor s_j \rfloor_h} = 2\big(\lfloor y_j \rfloor_{2 \sqrt{h}} - n\big) + \tau \big(m - \lfloor s_j \rfloor_h \big), j = 1, \dots, k ;\\
   &\hspace{6.7em} \text{and } S_m = \tau (m) - 2n\Big). 
   \end{align*}
Recall that $\ks_j$'s and $\ky_j$'s are defined as in \eqref{def_ks} and \eqref{def_ky}, respectively, and $S$ is a homogeneous Markovian chain with symmetric transition probabilities. It is easy to verify that $g_{h, k}$ in the above equation can also be expressed as in \eqref{def_gk-delta}. The proof of this lemma is complete. 
\end{proof}

\begin{proof}[Proof of Lemma \ref{lmm-frac-g-p-delta}]
Without loss of generality, assume that both $\lfloor t \rfloor_h - \lfloor s \rfloor_h$ and $\lfloor s \rfloor_h$ are even integers. To prove this lemma, it suffices to estimate 
 \begin{align*}
  L_1\coloneqq & \Big\|\frac{1}{2\sqrt{h}} G \big(\lfloor t \rfloor_h - \lfloor s \rfloor_h, 2 (\lfloor x \rfloor_{2 \sqrt{h}} - \lfloor \cdot \rfloor_{2 \sqrt{h}})\big) \Big(\frac{1}{2\sqrt{h}} G \big(\lfloor s \rfloor_h, 2\lfloor \cdot \rfloor_{2 \sqrt{h}}\big) - p_{s} (\cdot) \Big)\Big\|_{\cH_*}^2, 
 \end{align*}
 and
 \begin{align*}
  L_2 \coloneqq & \Big\| \Big(\frac{1}{2\sqrt{h}} G \big(\lfloor t \rfloor_h - \lfloor s \rfloor_h, 2 (\lfloor x \rfloor_{2 \sqrt{h}} - \lfloor \cdot \rfloor_{2 \sqrt{h}})\big) - p_{t - s} (x - \cdot)\Big)  p_{s} (\cdot) \Big\|_{\cH_*}^2.
 \end{align*}

\noindent \textbf{Estimate for $L_1$.~} In order to estimate $L_1$, we need to analyse the following cases. 

 \textit{Case i)~} Assume $t - s < h$. Then, 
it follows from Lemma \ref{lmm-frac-g-p} and the fact $G \leq 1$ as a probability mass function, that
\begin{align*}
 L_1 \lesssim h^{-1} \times s^{- [2H \wedge (2 - H_*)] + \epsilon/2} h^{(2 H + H_* - 1) \wedge 1 - \epsilon/2}, 
\end{align*}
with $\epsilon \in (0, (2H + H_* - 1) \wedge 1)$. Recall the assumption that $t - s < h$. Thus, 
\begin{align*}
 L_1 \lesssim s^{- [2H \wedge (2 - H_*)] + \epsilon/2} (t - s)^{- 1 + \epsilon/2} h^{(2 H + H_* - 1) \wedge 1 - \epsilon}. 
\end{align*}

 \textit{Case ii)~} In case $s < h \leq t - s$, we can write
\begin{align*}
 L_1 \lesssim L_{1, 1} + L_{1, 2}, 
\end{align*}
where
 \begin{align*}
  L_{1, 1} \coloneqq & \Big\|\frac{1}{2\sqrt{h}} G \big(\lfloor t \rfloor_h - \lfloor s \rfloor_h, 2 (\lfloor x \rfloor_{2 \sqrt{h}} - \lfloor \cdot \rfloor_{2 \sqrt{h}})\big) \times \frac{1}{2\sqrt{h}} G \big(\lfloor s \rfloor_h, 2\lfloor \cdot \rfloor_{2 \sqrt{h}}\big)\Big)\Big\|_{\cH_*}^2,  
  \end{align*}
 and
  \begin{align*}
  L_{1, 2} \coloneqq & \Big\|\frac{1}{2\sqrt{h}} G \big(\lfloor t \rfloor_h - \lfloor s \rfloor_h, 2 (\lfloor x \rfloor_{2 \sqrt{h}} - \lfloor \cdot \rfloor_{2 \sqrt{h}})\big) p_{s} (\cdot) \Big)\Big\|_{\cH_*}^2.
 \end{align*}
Due to Corollary \ref{coro_g-p} and Lemma \ref{lmm_g-h1} and the fact that 
\[
G \big(\lfloor s \rfloor_h, 2\lfloor x \rfloor_{2 \sqrt{h}}\big) = G \big(0, 2\lfloor x \rfloor_{2 \sqrt{h}}\big) = \1_{(- 2\sqrt{h}, 2 \sqrt{h})} (x), 
\]
we have
\begin{align*}
 L_{1, 1} \lesssim & p_{t - s} (0)^2 \Big\| \frac{1}{2\sqrt{h}} G \big(0, 2\lfloor \cdot \rfloor_{2 \sqrt{h}}\big)\Big)\Big\|_{\cH_*}^2 + \frac{h}{(t - s)^2} \Big\| \frac{1}{2\sqrt{h}} G \big(0, 2\lfloor \cdot \rfloor_{2 \sqrt{h}}\big)\Big)\Big\|_{\cH_*}^2 \\
 \lesssim & (t - s)^{-1} h^{H_* - 1} + (t - s)^{-2} h^{H_*}. 
\end{align*} 
Recall the assumption that $s < h \leq t - s$. It follows that
\begin{align*}
 L_{1, 1} \lesssim (t - s)^{-1} h^{H_* - 1} \lesssim s^{- [2H \wedge (2 - H_*)] + \epsilon/2} (t - s)^{- 1 + \epsilon/2}h^{(2 H + H_* - 1) \wedge 1 - \epsilon}. 
\end{align*}
In addition, by applying Corollary \ref{coro_g-p} once again, we deduce that
\begin{align*}
 L_{1, 2} \lesssim & \frac{h}{(t - s)^2}\| p_s\|_{\cH_*}^2 + \|p_{t - s} (x - \cdot) p_s(\cdot)\|_{\cH_*}^2 . 
 \end{align*}
Thus, Lemmas \ref{lmm_inn_heat} and \ref{lmm_heat-fac} imply that
 \begin{align*}
 L_{1, 2} \lesssim & s^{H_* - 1} (t - s)^{-2} h + p_{t}(x)^2 \Big\|p_{\frac{s (t - s)}{t}} \Big(\frac{t - s}{t} x - \cdot\Big)\Big\|_{\cH_*}^2\\
 \lesssim & s^{H_* - 1} (t - s)^{-2} h + s^{H_* - 1} (t - s)^{H_* - 1} t^{- H_*}\\
  \lesssim & s^{- [2H \wedge (2 - H_*)] + \epsilon/2} (t - s)^{- 1 + \epsilon/2} h^{(2 H + H_* - 1) \wedge 1 - \epsilon}. 
\end{align*}
because $s < h \leq t - s < t$. 
Hence, we conclude that 
\begin{align*}
 L_1 \lesssim s^{- [2H \wedge (2 - H_*)] + \epsilon/2} (t - s)^{- 1 + \epsilon/2} h^{(2 H + H_* - 1) \wedge 1 - \epsilon}. 
\end{align*}

 \textit{Case iii)~} When $s \wedge (t - s) \geq h$, as a result of Corollary \ref{coro_g-p} and Lemma \ref{lmm-frac-g-p}, 
\begin{align*}
 L_1 \lesssim L_{1, 1} + L_{1, 2}, 
\end{align*}
where
\begin{align*}
 L_{1, 1} \coloneqq & \frac{h}{(t - s)^2} \Big\|\frac{1}{2\sqrt{h}} G \big(\lfloor s \rfloor_h, 2(\lfloor \cdot \rfloor_{2 \sqrt{h}} \big) - p_{s} (\cdot)\Big\|_{\cH_*}^2\\
\lesssim & h (t - s)^{-2} \times s^{- [2H \wedge (2 - H_*)] + \epsilon/2} h^{(2 H + H_* - 1) \wedge 1 - \epsilon/2} \\
\lesssim & s^{- [2H \wedge (2 - H_*)] + \epsilon/2} (t - s)^{-1 + \epsilon/2} h^{(2 H + H_* - 1) \wedge 1 - \epsilon}, 
\end{align*}
and
\begin{align*}
 L_{1, 2} \coloneqq & \Big\| p_{t - s} (x - \cdot) \times \Big|\frac{1}{2\sqrt{h}} G \big(\lfloor s \rfloor_h, 2(\lfloor \cdot \rfloor_{2 \sqrt{h}} \big) - p_{s} (\cdot)\Big| \Big\|_{\cH_*}^2\\
\lesssim & (t - s)^{-1} \times s^{- [2H \wedge (2 - H_*)] + \epsilon/2} h^{(2 H + H_* - 1) \wedge 1 - \epsilon/2} \\
\lesssim & s^{- [2H \wedge (2 - H_*)] + \epsilon/2} (t - s)^{-1 + \epsilon/2} h^{(2 H + H_* - 1) \wedge 1 - \epsilon}. 
\end{align*}
Thus, 
\[
 L_1 \lesssim s^{- [2H \wedge (2 - H_*)] + \epsilon/2} (t - s)^{-1 + \epsilon/2} h^{(2 H + H_* - 1) \wedge 1 - \epsilon}, 
\]
when $s \wedge (t - s) \geq h$. Notice that under the assumption $H \in [\frac{1}{2}, 1)$ and $H_* \in [\frac{1}{2}, 1)$ ensures $2H \wedge (2 - H_*) \geq 1$. Hence, when $0 < s < t \leq T$, it holds for all Cases i)--iii) that
\begin{align}\label{ieq_g-p-delta-1}
 L_1 \lesssim T^{2H \wedge (2 - H_*) - 1} \big(s (t - s) \big)^{- 2H \wedge (2 - H_*)+ \epsilon/2} h^{(2 H + H_* - 1) \wedge 1 - \epsilon}. 
\end{align}
\noindent \textbf{Estimate for $L_2$.~} The estimate for $L_2$ is similar, so we outline only the crucial steps. 

\textit{Case i)~} In case $t - s < h$, due to Lemma \ref{lmm_heat-fac} we see that
\begin{align*}
 L_2 \lesssim & \Big\|\frac{1}{2\sqrt{h}} G \big(\lfloor t \rfloor_h - \lfloor s \rfloor_h, 2(\lfloor \cdot \rfloor_{2 \sqrt{h}} \big) p_{s} (\cdot) \Big\|_{\cH_*}^2 + \big\| p_{t - s} (x - \cdot) p_{s} (\cdot) \big\|_{\cH_*}^2\\
 \lesssim & \big(h^{-1} \|p_s\|_{\cH_*}^2\big) \wedge \Big(s^{-1}\Big\|\frac{1}{2\sqrt{h}} G \big(\lfloor t \rfloor_h - \lfloor s \rfloor_h, 2(\lfloor \cdot \rfloor_{2 \sqrt{h}} \big)\Big\|_{\cH_*}^2\Big) + s^{H_* - 1} (t - s)^{H_* - 1} t^{- H_*}. 
\end{align*}
Then, as a result of Lemmas \ref{lmm_g-h1} and \ref{lmm_inn_heat}, we can write
\begin{align*}
 L_2 \lesssim & \begin{dcases}
 h^{-1} s^{H_* - 1} + s^{- 1 + \epsilon/2} (t - s)^{H_* - 1 - \epsilon/2}, & s < h, \\
 s^{-1} (t - s)^{H_* - 1} + s^{- 1 + \epsilon/2} (t - s)^{H_* - 1 - \epsilon/2}, & s \geq h. 
 \end{dcases}
\end{align*}
In both cases, with the assumption $t - s < h$, it holds that
\[
 L_2 \lesssim s^{-1 + \epsilon/2} (t - s)^{- 2H \wedge (2 - H_*)+ \epsilon/2} h^{(2 H + H_* - 1) \wedge 1 - \epsilon}. 
\]

 \textit{Case ii)~} If $s < h \leq t - s$, simply using Corollary \ref{coro_g-p}, we can write
\begin{align*}
 L_2 \lesssim h (t - s)^{-2} \|p_s (\cdot)\|_{\cH_*}^2
 \lesssim s^{- [2H \wedge (2 - H_*)] + \epsilon/2} (t - s)^{- 1 + \epsilon/2} h^{(2 H + H_* - 1) \wedge 1 - \epsilon}. 
\end{align*}

 \textit{Case iii)~} When $h \leq s \wedge(t - s)$, it follows from Lemma \ref{lmm-frac-g-p} that
\begin{align*}
 L_2 \lesssim & s^{-1} \Big\| \frac{1}{2\sqrt{h}} G \big(\lfloor t \rfloor_h - \lfloor s \rfloor_h, 2 (\lfloor x \rfloor_{2 \sqrt{h}} - \lfloor \cdot \rfloor_{2 \sqrt{h}})\big) - p_{t - s} (x - \cdot)\Big\|_{\cH_*}^2\\
 \lesssim & s^{-1 + \epsilon/2} (t - s)^{- 2H \wedge (2 - H_*)+ \epsilon/2} h^{(2 H + H_* - 1) \wedge 1 - \epsilon}. 
\end{align*}
Hence, we conclude that under the assumption of Lemma \ref{lmm-frac-g-p-delta}, the next inequality holds, 
 \begin{align}\label{ieq_g-p-delta-2}
 L_2 \lesssim  \big(s (t - s) \big)^{- 2H \wedge (2 - H_*)+ \epsilon/2} h^{(2 H + H_* - 1) \wedge 1 - \epsilon}. 
\end{align}
The proof of this lemma is complete as a consequence of inequalities \eqref{ieq_g-p-delta-1} and \eqref{ieq_g-p-delta-2}. 

\end{proof}

\begin{proof}[Proof of Lemma \ref{prop_f-g-delta}]
If $k = 1$, this lemma is ensured by Lemma \ref{lmm-frac-g-p-delta}. For general $k \geq 2$, similarly to the proof of Lemma \ref{prop_l2-f-g}, we decompose 
\begin{align*}
 \big\| f_k (\bfs_k, \cdot) & - g_{h, k} (\bfs_k, \cdot) \big\|_{\cH_*^{\otimes k}}^2 \lesssim J_1 + J_2 , 
\end{align*}
 where, by Cauchy-Schwarz's inequality, 
 \begin{align*}
  J_1 \coloneqq \int_{\bR^2} d y_k d z_k & \frac{1}{4h} |y_k - z_k|^{2H_* - 2}G (\ks_k, \ky_k) G (\ks_k, \kz_k) \big\| f_{k - 1}^{s_k, y_k} (\bfs_{k - 1}, \cdot) - g_{k - 1, h}^{s_k, y_k} (\bfs_{k - 1}, \cdot)\big\|_{\cH_*^{k - 1}} \\
& \times \big\| f_{k - 1}^{s_k, z_k} (\bfs_{k - 1}, \cdot) - g_{k - 1, h}^{s_k, z_k} (\bfs_{k - 1}, \cdot)\big\|_{\cH_*^{k - 1}}, 
 \end{align*} 
 and
 \begin{align*}
  J_2 \coloneqq & \int_{\bR^2} d y_k d z_k |y_k - z_k|^{2H_* - 2} \big\|f_{k - 1}^{s_k, y_k}(\bfs_{k - 1}, \cdot) \big\|_{\cH_*^{\otimes (k - 1)}} \big\|f_{k - 1}^{s_k, z_k}(\bfs_{k - 1}, \cdot) \big\|_{\cH_*^{\otimes (k - 1)}} \\
& \qquad \times \Big( \frac{1}{2\sqrt{h}} G (\ks_k, \ky_k) - p_{t - s_{ k}} (x - y_{k}) \Big) \Big( \frac{1}{2\sqrt{h}} G (\ks_k, \kz_k) - p_{t - s_{ k}} (x - z_{ k}) \Big) . 
 \end{align*} 
Then, it follows from the induction hypothesis and Lemma \ref{lmm_g-h1} that
\begin{align}\label{f-g-frac-delta-1}
 J_1 \leq C^{k} & h^{(2H + H_* - 1) \wedge 1 - \epsilon} \bigg( \sum_{j = 1}^{k - 1} \Big[ (s_{j + 1} - s_j)^{-[2H \wedge (2 - H_*)]+ \epsilon} s_j^{-H_*}\prod_{\substack{1 \leq i \leq k \\ i \neq j}} (s_{i + 1} - s_i)^{H_* - 1} \Big] \nonumber \\
& \qquad \qquad + \prod_{i = 2}^k (s_{i + 1} - s_i)^{H_* - 1} (s_2 - s_1)^{-[2H \wedge (2 - H_*)] + \epsilon/2} s_1^{-[2H \wedge (2 - H_*)] + \epsilon/2} \bigg). 
\end{align}
Additionally, Lemma \ref{lmm-frac-g-p} and inequality \eqref{eq_est-chaos-delta} suggests that
\begin{align}\label{f-g-frac-delta-2}
  J_2 \leq C^{k} h^{(2H + H_* - 1) \wedge 1 - \epsilon} (t - s_k)^{-[2H \wedge (2 - H_*)]+ \epsilon} s_k^{-H_*} \prod_{i = 0}^{k - 1} (s_{i + 1} - s_i)^{H_* - 1}. 
\end{align}
Then, \eqref{f-g-frac-delta} is a direct result of \eqref{f-g-frac-delta-1} and \eqref{f-g-frac-delta-2}. The proof of this lemma is complete. 
\end{proof}

\section{Another look from the random polymer viewpoint}\label{sec_dpre}
The mathematical model of directed polymers in random environments (DPREs) is introduced to study the behaviour of a polymer chain in a disordered medium. This model has wide-ranging applications across various fields, including statistical physics, condensed matter physics, DNA and protein folding, geophysics, climate modelling, and so on. For a comprehensive treatment of the topic, we refer the reader to the monograph \cite{springer-17-comets}. 

Roughly speaking, the DPRE can be modelled by a simple random walk $S = \{S_m \colon m \in \bZ_{\geq 0}\}$ in a random environment $\omega = \{\omega(m, n) \colon (m, n) \in \bZ_{\geq 0} \times \bZ\}$ satisfying the exponential integrability condition. The model is connected to \eqref{pam} through the `partition function', which acts as a renormalising constant to ensure that the polymer measure is a probability measure. This partition function $Z (\beta) = \{Z_m (\beta) \colon m \in \bZ_{\geq 0}\}$ can be expressed as follows
\begin{align*}
 Z_m(\beta) \coloneqq \bE^S \bigg[\exp \Big(\beta \sum_{1\leq j\leq m} \omega(j, S_j) \Big)\bigg], 
\end{align*}
where $\beta > 0$ represents the inverse temperature. Assuming that $\omega (m, n)$ are i.i.d. random variables, it was showed in \cite{ap-14-alberts-khanin-quastel} that under a diffusive scaling, the rescaled partition function
\begin{align*}
 e^{- m \lambda (\beta m^{-\frac{1}{4}})} Z_m \big(\beta m^{-\frac{1}{4}}\big), 
\end{align*}
converges in distribution to $u(1, 0)$ as $m \to \infty$, where $\lambda (\beta) \coloneqq \log \bE^{\omega} [e^{\beta \omega(1, 1)}]$, and $u$ is the solution to \eqref{pam} driven by space-time white noise with a flat initial condition and the diffusion term replaced by $\sqrt{2}\beta u \diamond \dot{W}$. This problem was further studied in \cite{jems-16-caravenna-sun-zygouras}. In this paper, a generalised Lindeberg principle, building on the work \cite{am-10-mossel-odonnell-krzysztof}, was developed, yielding general conditions for convergence. Recently, similar questions have been investigated with temporally or spatially correlated random environments, and with non-Markovian random walks; cf. \cite{ejp-23-chen-gao, spa-20-rang, ejp-24-rang-song-wang}. 

Since in general the environments in \eqref{eq_cov-wh} is assumed to be only exponentially integrable, one of the challenges in the aforementioned works lies in the (weak) convergence of polynomial chaos to Wiener chaos. However, if we assume the environment to be Gaussian, as discussed in this article, polynomial chaos can be interpreted as Wiener chaos, with an appropriate coupling if necessary. This simplification enables us to bypass the difficulties encountered in existing papers, and obtain some quantitative results. 

Let $S$ be a simple random walk, and let 
$\omega = \{\omega(m, n) \colon (m, n) \in (\bZ_{\geq 0} \times \bZ)^* \}$, 
independent of $S$, be a Gaussian random field, 
where 
\begin{align*}
(\bZ_{\geq 0} \times \bZ)^* \coloneqq \big\{(m, n) \subseteq \bZ_{\geq 0} \times \bZ \colon m + n \text{ is an even integer}. \big\}. 
\end{align*}
The correlation of $\omega$ can be expressed as follows
\begin{align}\label{cor_omega}
 \bE \big[\omega (m_1, n_1) \omega (m_2, n_2)\big] = \Gamma (m_1 - m_2) \Gamma_* \big(\lfloor n_1/2 \rfloor - \lfloor n_2/2 \rfloor \big), 
\end{align}
where $\Gamma$ and $\Gamma_*$ are defined by \eqref{def_gmm-t} and \eqref{def_gmm*-s}, respectively.

Let $h > 0$ and let $W_h$ be a Gaussian random field defined in \eqref{def_wh}. 
Additionally, for any $h > 0$, $m \in \bZ_{> 0}$, $n \in \bZ$, and $k \in \{1, \dots, m\}$, define
\[
 \omega_h ^{(m)}(k, n) \coloneqq 2^{2H_* - 1}h^{\frac{1}{2}(2H + H_* - 1)} \omega \big(m + 1 - k, 2n + \tau (m + 1 - k)\big), 
\]
where $\tau$ is defined in \eqref{def_tau-t}. It can be easily deduced from \eqref{eq_cov-wh} that
$\{\omega_h ^{(m)}(k, n) \colon (k, n) \in \{1, \dots, m\} \times \bZ \}$
shares the same distribution as $\{W_h(k, n) \colon (k, n) \in \{1, \dots, m\} \times \bZ \}$. 

Let  $S$ be a simple random walk, and let $\omega$ be a Gaussian random filed with correlation given by \eqref{cor_omega}, and define the following (Wick renormalised) partition function
\begin{align}\label{def_hat-z}
 \cZ_m \coloneqq \bE^S \bigg[ & \exp \bigg( - 2^{4H_* - 3} m^{1 - 2H - H_*} {\rm Var}^{\omega} \Big( \sum_{1\leq j\leq m} \omega(j, S_j)\Big) \bigg) \nonumber \\ 
 & \times \exp \Big(2^{2H_* - 1} m^{-\frac{1}{2}(2H + H_* - 1)} \sum_{1\leq j\leq m} \omega(j, S_j) \Big)\bigg].
 \end{align}
One can show that
 \begin{align*}
  \cZ_m \overset{\rm (d)}{=} &\bE^S \bigg[ \exp \bigg( \sum_{1\leq j\leq m} W_{1/m} \big( m + 1 - j, \lfloor S_j/2 \rfloor \big) - \frac{1}{2} {\rm Var}^{W} \Big( \sum_{1\leq j\leq m} W_{1/m} \big(m + 1 - j, \lfloor S_j/2 \rfloor \big) \Big) \bigg) \\
 = & u_{1/m}(1, 0), 
\end{align*}
where $u_{1/m}$ is defined in \eqref{fk_h_X}. 

Additionally, for any $p \geq 2$, let $W_p$ denote the $p$-Wasserstein distance on the space of probability measures. Specifically, for any probability measure $\mu$ and $\nu$ on $\bR$, $W_p (\mu, \nu)$ is define by
\[
	W_p (\mu, \nu) \coloneqq \inf \big\{ \|X - Y\|_{L^p (\Omega)} \colon {\rm law} (X) = \mu \quad \text{and} \quad {\rm law} (Y) = \nu \big\};
\]
cf. \cite[Deﬁnition 6.1]{springer-09-villani}. By abuse of notation, we write $W_p (X, Y) \coloneqq W_p ({\rm law} (X), {\rm law} (Y))$. Therefore, with regard to Theorem \ref{thm_rate-frac}, we can immediately state the following theorem. 

\begin{theorem}\label{thm_dprm-flat}
 For every $p \geq 2$, $m \in \bZ_{> 0}$ and $\epsilon \in (0, \frac{1}{2}[(2 H + H_* - 1)\wedge 1])$, with a universal implicit constant, 
 \[
  W_p \big( \cZ_m, u(1, 0)\big) \lesssim m^{ - \frac{1}{2}[(2 H + H_* - 1)\wedge 1] + \epsilon}, 
 \]
where $ \cZ_m$ is given by \eqref{def_hat-z}, $u$ is the solution to \eqref{pam} with a flat initial condition.
\end{theorem}

Replacing the random walk with a random walk bridge, we can similarly get the next theorem as a result of Theorem \ref{thm_rate-delta}, subject to \eqref{pam} with a delta initial condition. 

\begin{theorem}\label{thm_dprm-delta}
Let $u$ denote the solution to \eqref{pam} with a delta initial condition at the origin, and let $\omega$ be a Gaussian random field with correlation specified by \eqref{cor_omega}. For any $m \in \bZ_{> 0}$, let $\bS^{(m, \tau (m))}$ be a simple random walk bridge, independent of $\omega$, with $\bS^{(m, 0)}_0 = \bS_m^{(m, \tau (m))} = \tau (m)$, where $\tau$ is defined in \eqref{def_tau-t}; and let the (Wick renormalised) partition function
\begin{align*}
 \cZ_m \coloneqq  \frac{\sqrt{m}}{2} G \big(m, \tau(m)\big)  \bE^{\bS} \bigg[& \exp \bigg( - 2^{4H_* - 3} m^{1 - 2H - H_*} {\rm Var}^{\omega} \Big( \sum_{1\leq j\leq m} \omega \big( j, \bS^{(m, \tau (m))}_j\big)\Big) \bigg) \nonumber \\
 &  \times\exp \Big(2^{2H_* - 1} m^{-\frac{1}{2}(2H + H_* - 1)} \sum_{1\leq j\leq m} \omega \big(j, \bS^{(m, \tau (m))}_j \big) \Big)\bigg]. 
 \end{align*}
Then, for all $p \geq 2$, $m \in \bZ_{> 0}$, and $\epsilon \in (0, \frac{1}{2}[(2 H + H_* - 1)\wedge 1])$, with a universal implicit constant, 
 \[
  W_p \big( \cZ_m, u(1, 0)\big) \lesssim m^{ - \frac{1}{2}[(2 H + H_* - 1)\wedge 1] + \epsilon}. 
 \]
\end{theorem}

 In both Theorems \ref{thm_dprm-flat} and \ref{thm_dprm-delta}, we consider only the case when $t = 1$ and $x = 0$. They can likely be extended to general $(t, x) \in \bR_{> 0} \times \bR$ with minor adjustments. However, since this is not the focus of the current paper, we choose not to pursue this extension further.





\end{document}